\documentclass{imsart}

\RequirePackage{amsthm,amsmath,amsfonts,amssymb}
\RequirePackage[numbers]{natbib}
\RequirePackage[colorlinks,citecolor=blue,urlcolor=blue]{hyperref}
\RequirePackage{graphicx}
\RequirePackage{float}
\RequirePackage{xcolor}
\RequirePackage{rotating}
\RequirePackage[format=plain,font=it,labelfont=it]{caption}
\RequirePackage{enumitem}
\RequirePackage[normalem]{ulem}
\RequirePackage{verbatim}
\RequirePackage{tikz}
\usetikzlibrary{arrows.meta,positioning}
\RequirePackage{subcaption}
\RequirePackage{booktabs}
\RequirePackage[export]{adjustbox}

\startlocaldefs

\theoremstyle{plain}

\newtheorem{theorem}{Theorem}[section]
\newtheorem{lemma}[theorem]{Lemma}
\newtheorem{remark}[theorem]{Remark}
\newtheorem{corollary}[theorem]{Corollary}
\newtheorem{proposition}[theorem]{Proposition}

\theoremstyle{definition}
\newtheorem{definition}[theorem]{Definition}

\endlocaldefs

\begin{document}

\begin{frontmatter}
\title{A small noise approximation for Muller's Ratchet}
\runtitle{A small noise approximation for Muller's Ratchet}

\begin{aug}
%%%%%%%%%%%%%%%%%%%%%%%%%%%%%%%%%%%%%%%%%%%%%%%
%% Only one address is permitted per author. %%
%% Only division, organization and e-mail is %%
%% included in the address.                  %%
%% Additional information such as            %%
%% identifying the corresponding author must %%
%% be included in in the Acknowledgments     %%
%% section if necessary.                     %%
\author[A]{\fnms{Carola Sophia }~\snm{Heinzel} \ead[label=e1]{carola.hienzel@stochastik.uni-freiburg.de }\orcid{0009-0000-9042-949X}}
\author[A]{\fnms{Peter}~\snm{Pfaffelhuber}\ead[label=e2]{p.p@stochastik.uni-freiburg.de}\orcid{0000-0002-6421-5460}}
\author[B]{\fnms{Anton}~\snm{Wakolbinger}\ead[label=e3]{wakolbin@math.uni-frankfurt.de}}
%%%%%%%%%%%%%%%%%%%%%%%%%%%%%%%%%%%%%%%%%%%%%%
%% Addresses                                %%
%%%%%%%%%%%%%%%%%%%%%%%%%%%%%%%%%%%%%%%%%%%%%%
\address[A]{Department of Mathematical Stochastics,
University of Freiburg\printead[presep={ ,\ }]{e1}}

\address[B]{Department of Mathematics,
University of Frankfurt \printead[presep={,\ }]{e2,e3}}
\end{aug}

\begin{abstract}
We consider an infinite system of SDEs with Fleming-Viot noise indexed by $k=0,1,2,...$, whose parameters $\alpha,\lambda$, and $\nu$ are the (deleterious) selection coefficient, the (uni-directional) mutation rate, and a quantity which determines the size of the system’s fluctuations.   The SDE's unique weak solution $X(t) = (X_k(t))_{k=0,1,2,...}$ models what is known in population genetics as \emph{Muller's ratchet}. Here, $X_k(t)$ stands for the frequency of individuals carrying $k$ deleterious mutations. Since the mutation process is uni-directional, $t\mapsto \inf\{k: X_k(t)> 0\}$ is non-decreasing for almost every path of $X$, and we refer to an increase as a click of Muller's ratchet. A long standing question concerns the clicking rate of Muller's ratchet. Using Duhamel's principle for semigroups, we give a partial answer by approximating $\mathbb E\Big(\sum_{k=1}^\infty kX_k(t) \Big)$ and $\mathbb E\big(X_0(t)\big)$ up to $\mathcal O(1/\nu^2)$ for fixed $\alpha$, $\lambda$ and $t>0$. Our results suggest that  $\psi:=\nu \alpha e^{-\lambda/\alpha}$ is a crucial quantity  also when the mutation/selection ratio $\theta = \lambda/\alpha$ is moderately large:   for large $\nu \alpha$, clicking of the ratchet on the time scale $\tfrac 1\alpha \log \theta$ becomes rare as soon as $\psi$ becomes large.
\end{abstract}



\end{frontmatter}
%%%%%%%%%%%%%%%%%%%%%%%%%%%%%%%%%%%%%%%%%%%%%%
%% Please use \tableofcontents for articles %%
%% with 50 pages and more                   %%
%%%%%%%%%%%%%%%%%%%%%%%%%%%%%%%%%%%%%%%%%%%%%%
%\tableofcontents

\section{Introduction and main results}
\label{S:intro}
In \cite{Muller1964}, Muller has introduced the following genetic model: A population of fixed size~$N$ reproduces randomly. Each individual carries a number of mutations, all of which are assumed to be deleterious. Fitness (measured by the expected number of offspring) of an individual carrying $k$ mutations is proportional to $(1 - \mathfrak s)^k$ for some $\mathfrak s\geq 0$. The offspring of an individual has a small chance $\mathfrak u \geq 0$ to gain a new deleterious mutation. In particular, offspring of an individual has at least as many mutations as the parent and mutation is an irreversible process. Hence, there is a chance that, eventually, all individuals will have at least one mutation and we say that the ratchet has clicked. So, mutations drive the population to a larger number of deleterious mutations while selection acts in the opposite direction. This means that Muller's ratchet describes a form of mutation-selection quasi-balance and the question is if selection is strong enough in order to prevent the accumulation of deleterious mutations, at least for some time. In the biological literature, Muller's ratchet is a textbook model for the evolutionary advantage of reproduction under recombination \citep{Felsenstein1974}, has many different applications  (e.g. \cite{mcfarland2013impact, howe2008muller}) and huge simulation studies try to answer how fast fitness of a population decreases by Muller's ratchet \citep{Loewe2006}. 

Variants of Muller's ratchet have been studied in the past years. \cite{pfaffelhuber2012muller} allowed for compensatory back-mutations, and \cite{foutel2020spatial} considered a spatial version of the ratchet. \cite{gonzalez2023quasi} considered a modification of the fitness function leading to the so-called tournament ratchet: here, selection between individuals of type $k$ and of type $\ell$ only depends on whether $k > \ell$ or $k < \ell$ and not on the difference $k-\ell$. With this fitness function, $(X_0, ..., X_n)$ is an autonomous process for any $n \in \mathbb N$, which makes the analysis of the tournament ratchet more accessible than that of the classical one. Recently,~\cite{ISW2025} managed to compute the tournament ratchet's click rate for a parameter regime leading to metastability.

\subsection{A Fleming-Viot model of Muller's ratchet}
Here, we study a version of Muller's ratchet described by the infinite system of stochastic differential equations
\begin{align}\label{eq:ratchetSDE}
  d X_k = \Big( \alpha \Big(\sum_{j=0}^\infty (j-k)X_j\Big) X_k + \lambda(X_{k-1} - X_k)\Big)dt + \sum_{\ell \neq k}\sqrt{\frac 1{\nu} X_k X_\ell} \, dW_{k\ell}, 
\end{align}
for $k=0,1,\ldots$ with $X_{-1}:=0$, $\sum X_k=1$, $\lambda\geq 0$ (the mutation rate), $\alpha\geq 0$ (the selection coefficient) and $\nu>0$ (the effective population size on the chosen time-scale), where $(W_{k\ell})_{k>\ell}$ is a family of independent Brownian motions and $W_{k\ell} = -W_{\ell k}$. Existence and uniqueness of a (weak) solution of \eqref{eq:ratchetSDE} was shown in \cite{pfaffelhuber2012muller} (under the assumption of exponential moments) and \cite{audiffren2013muller} (under the weaker assumption of moments of order $2 + \delta$). The system~\eqref{eq:ratchetSDE} and the Wright-Fisher (WF)  model of Muller's ratchet  (sketched above and described in more detail in Sec~\ref{simul}) are ``synchronized'' by measuring time $t$ in ~\eqref{eq:ratchetSDE} in units of WF generations and putting $(\alpha, \lambda, \nu):= (\mathfrak s, \mathfrak u, N)$; with this choice, expectation and covariance of the increments of $X$ in ~\eqref{eq:ratchetSDE} mimic those of the type frequencies in the WF system.

One of the benefits of \eqref{eq:ratchetSDE} is its time-scaling property: With $X^{(\alpha, \lambda, \nu)}$ denoting the solution of \eqref{eq:ratchetSDE}, one has for all $c>0$
\begin{equation}\label{eq:scaling}
    \big(X^{(\alpha, \lambda, \nu)}(tc)\big)_{t\ge 0}  \stackrel d= 
     \big(X^{(c\alpha, c\lambda, \nu /c)}(t)\big)_{t\ge 0}.
\end{equation}
Thus, measuring time $t$ in~\eqref{eq:ratchetSDE} in $N$ generations of the WF system would lead to the choice $(\alpha, \lambda, \nu) := (\mathfrak s N, \mathfrak u N, 1)$. Below, we will work with yet another timescale, turning $(\alpha, \lambda, \nu)$ into $(1, \lambda/\alpha, \nu \alpha)$,  which is appropriate in the case of {\em moderate selection}, i.e. $1/N \ll \mathfrak s \ll 1$ as $N\to \infty$.

 The {\em mutation-selection ratio} 
$$\theta:= \frac \lambda \alpha$$
will play a key role  for the  frequency of clicks of \eqref{eq:ratchetSDE}.  If $X = (X(t))_{t\geq 0}$, $X(t) = (X_k(t))_{k=0,1,2,...}$ is the solution of \eqref{eq:ratchetSDE}, let
\begin{align}\label{eq:Kast}
  K^\ast(t) := \inf\{k\in\mathbb N_0: X_k(t)>0\},
\end{align}
which we refer to as the \emph{fittest class} (or \emph{least-loaded class}) by time $t$, because selection favors individuals with a small number of deleterious mutations. Since mutation is irreversible, $t\mapsto K^\ast(t)$ is non-decreasing, and any time $\tau>0$ with $K^\ast(\tau) - K^\ast(\tau-)>0$ is a \emph{click} of Muller's ratchet. In \cite{audiffren2013muller} it was proved  that the ratchet clicks a.s. in finite time, and \cite{mariani2025metastability} \mbox{analysed} properties of quasistationarity of $X$. We believe (even though we are not aware of a proof) that for all $(\alpha, \lambda, \nu)$  both the a.s. limit  of $\tfrac 1t K^\ast(t)$ and that of $\tfrac 1t \sum_k k X_k(t)$ (as $t\to \infty$) exist and are equal to some constant $\mathfrak r =  \mathfrak r (\alpha, \lambda, \nu)$, independent of the initial condition $X(0)$.

A formula for $\mathfrak r(\alpha, \lambda, \nu)$ seems out of reach. However, $\mathfrak r$ (provided it exists) inherits from~\eqref{eq:scaling} the scaling property 
$$\mathfrak r(\alpha, \lambda, \nu) = \alpha\,  \mathfrak r(1, \theta, \nu\alpha),$$
i.e.\ $\mathfrak r/\alpha $ only depends on $\theta$ and $\nu\alpha$ rather than on the full triple $(\alpha, \lambda, \nu)$; see also Remark~\ref{rem:scaling} below. 

~

In the case $\nu=\infty$, the system \eqref{eq:ratchetSDE} has an explicit solution; see Proposition~4.1 in \cite{EtheridgePfaffelhuberWakolbinger2007}, Theorem 2 in \cite{pfaffelhuber2012muller}, or Theorem~\ref{T2} below. In particular, if $X_0(0)>0$, the solution $X(t)$ of this system converges as $t\to \infty$ to $\pi := \text{Poi}(\theta)$, the Poisson distribution with parameter $\theta$. Thus, for $\nu=\infty$ one clearly has $K^\ast (t) = K^\ast (0)$.

Since the clicking rate $\mathfrak r$ is unavailable for finite $\nu$, but~\eqref{eq:ratchetSDE} is solvable for $\nu=\infty$, here we study asymptotic properties of $X$ as $\nu\to \infty$. In particular, as a first rigorous quantitative step in exploring the asymptotics of the rate of Muller's ratchet, for fixed $\alpha$, $\lambda$  and $t$  we derive in Corollary~\ref{cor} the expected values of quantities as 
$$m_1(X(t)) := \sum_{k=0}^\infty k X_k(t)$$ 
and $X_k(t)$,  and their asymptotics as $\nu \to\infty$, given the system is started in $\pi$. \subsection{Heuristic approaches}
We briefly recall the history of quantitative approaches for Muller's ratchet.  The first mathematical description was given by \cite{Haigh1978}. It was followed by some other investigations; see \cite{StephanChaoSmale1993}, \cite{Gessler1995}, \cite{HiggsWoodcock1995}, \cite{GordoCharlesworth2000}, \cite{RouzineWakeleyCoffin2003}, \cite{EtheridgePfaffelhuberWakolbinger2007}, \cite{pmid18689884}.
However, all of these papers rely on approximations in order to study the full system~\eqref{eq:ratchetSDE} or its discrete counterpart in the case of finite $\nu$. In particular, it has become common to try to approximate the full system by a one-dimensional diffusion for $X_0$ and studying the hitting time of 0. This is promising since $(X_0)_{t\geq 0}$ solves the SDE
\begin{align}
\label{eq:dX0}
dX_0 = (\alpha m_1(X) - \lambda )X_0 dt + \sqrt{\frac 1{\nu} X_0(1-X_0)}\, dW
\end{align}
for some Brownian motion $W$, provided that $X$ solves~\eqref{eq:ratchetSDE}.  Hence, given an approximate functional relationship $m_1(X) \approx \widehat m_1(X_0)$ leads to a closed system for $X_0$ and it is possible to study the resulting one-dimensional diffusion by classical theory. Using some heuristics and calculations, \cite{StephanChaoSmale1993}, \cite{CharlesworthCharlesworth1997}, \cite{EtheridgePfaffelhuberWakolbinger2007} suggest the approximate functional relationship
\begin{align}
  \label{eq:appEPW}
  \widehat m_1(X_0) = \theta + \kappa\Big(1 - \frac{X_0}{e^{-\theta}}\Big)
\end{align}
for $\kappa=0.6, \kappa=0.7$ and $\kappa=0.58$, respectively.
It was shown that this approximation gives reasonable results for the average time between clicks of the ratchet for various parameter combinations when compared to simulations. However -- see Table 6.1 of \cite{StephanKim2002} -- some other parameter combinations lead to large errors up to a factor of 2, even for large populations. \cite{assaf2011fixation, metzger2013distribution, brautigam2022diffusion} estimated the expected time of first click with different approximations, but still using a 1d-heuristics for the evolution of $X_0$. 
%\cite{waxman2010stochastic} considered a similar model from which they conclude an approximation of the rate of the ratchet for some parameter combinations. 
In \cite{neher2012fluctuations}, linear perturbations, an eigen-decomposition and path integrals are used for analysing the evolution of Muller’s ratchet. Breaking the analysis down to the 1d-case, they find a better approximation for the rate of the ratchet when letting $\kappa$ in 
\eqref{eq:appEPW} depend on~$\theta$. 

Before we start with our rigorous results for the solution $X$ of~\eqref{eq:ratchetSDE}, we add yet another heuristics for the rate of the ratchet, which involves the first two components of~$X$.

\begin{remark}[Yet another heuristics on the rate of the ratchet]
Given $D(t):= 1-X_0(t)-X_1(t)$, we find that
\begin{align*}
dX_0 & = (\alpha m_1(X)-\lambda)X_0\, dt + \sqrt{\frac 1{\nu} X_0X_1}dW_{01} + \sqrt{\frac 1{\nu} X_0D}dW_{0},\\
dX_1 & = \left((\alpha m_1(X) - \alpha-\lambda)X_1+\lambda X_0\right)dt - \sqrt{\frac 1{\nu} X_0X_1}dW_{01} + \sqrt{\frac 1{\nu} X_1D}dW_{1}.
\end{align*}
for independent Brownian motions, $W_0, W_1$, and $W_{01}$.  It\^o's lemma applied to $Q:=q(X_0,X_1):= X_0/X_1$ yields
\begin{align} \label{Ito}dQ & = \frac{1}{X_1}dX_0 - \frac{X_0}{X_1^2}dX_1 - \frac{1}{X_1^2} d\langle X_0, X_1\rangle +  \frac{X_0}{X_1^3} d\langle X_1\rangle \\
& = \notag 
Q(\alpha -\lambda Q)dt
 +\sqrt{\frac 1{\nu}} \left(\sqrt Q dW_{01} + \sqrt{\frac D{X_1}}\sqrt
  Q dW_0 + Q^{3/2}dW_{01} + \sqrt{\frac D{X_1}} Q dW_1\right) \\
\notag & \qquad \qquad \qquad \qquad \qquad \qquad \qquad \qquad \qquad \qquad + \frac 1{\nu} \left(Q^2+ \frac D{X_1}Q+Q\right)dt. 
\end{align}
The fluctuation term in \eqref{Ito} simplifies to
  \begin{align}\notag
    \sqrt Q(1+Q)dW_{01}+\sqrt{\frac D{X_1}}\sqrt Q dW_0+ \sqrt{\frac
      D{X_1}} Q dW_1=\sqrt{Q(1+Q)\left(\frac D{X_1} + 1+Q\right)} dW.
\end{align}
for some standard Brownian motion $W$. Since $\frac D{X_1} + 1+Q = \frac 1{X_1}$, \eqref{Ito} gets the interesting form
\begin{align}\label{dynQ}
dQ = Q\Big(\alpha-\lambda Q + \frac1{\nu X_1}\Big)dt + \sqrt{\frac {Q(1+Q)}{\nu X_1}} dW.
\end{align}
This means that the evolution of $Q$ only depends on $X_1$ (and $Q$). 
Now, set $\widetilde Q(t) := \theta Q(t/\alpha)$ and $\widetilde X_1(t) = X_1(t / \alpha)$. Then, a little  calculation gives
\begin{align}\label{sdynQ}
d\widetilde Q = \widetilde Q\Big(1-\widetilde Q + \frac1{\nu \alpha \widetilde X_1}\Big)dt +  \sqrt{\widetilde Q   \Big(1+\frac{\widetilde Q}{\theta}\Big) \frac {\theta}{\nu \alpha \widetilde X_1}}d\widetilde W
\end{align}
for a standard Brownian motion $\widetilde W$. Note that the last calculations do not require that $X(t)$ is close to $\pi$. However, observe that $\widetilde Q(t)=1$ if $X(t) = \pi$ and if $X(t)$ is close to $ \pi$, we see that the last term becomes $\sqrt{\widetilde Q\Big(1 + \frac{\widetilde Q}{\theta}\Big)\frac{1}{\psi}}\, d\widetilde W$, with $\psi = \nu \alpha e^{-\theta}$. Since we rescaled time by $1/\alpha$, this suggests that the click time of the ratchet is much larger than $1/\alpha$ if $\psi \gg 1$ and is of order $1/\alpha$ if $\psi = \mathcal O(1)$.
\end{remark}

~

\subsection{Main results} \label{secMR}
We assume throughout that $\alpha, \lambda, \nu > 0$. A key for deriving some interesting expectations will be the following result  for the semigroup of $X$, evaluated at products of the moment generating functions
\begin{align}\label{eq:gxi}
    g_\xi(x) := \sum_{k=0}^\infty x_k e^{-\xi k}.
\end{align}

\begin{theorem}[Towards the rate of the ratchet]
  \label{T1}
  Let $X$ be a solution of \eqref{eq:ratchetSDE} with $\alpha,  \lambda > 0$ and $X(0) = \text{Poi}(\theta)$. Then, for $\xi, \eta, t \geq 0$, 
  \begin{align*}
  \mathbb E & \big( g_\xi(X(t)) g_\eta(X(t))\big) \\ & = \exp\Big( -\theta(2 - e^{-\xi} - e^{-\eta})\Big) \\ & \cdot \Big(1 +  \frac{1}{\nu \alpha} \int_0^{t \alpha} \Big(3 \exp\Big(\theta(1 - e^{-s})^2 \Big) - 2 \exp\Big(\theta(1 - e^{- (s + \xi)})(1 - e^{-s}) \Big) \\ 
  & \qquad - 2 \exp\Big(\theta(1 - e^{-(s + \eta)})(1 - e^{-s}) \Big) + \exp\Big( \theta(1 - e^{-(s + \xi)})(1 - e^{-(s + \eta)})\Big)\Big) ds\Big) \\ & \qquad \qquad \qquad \qquad \qquad \qquad \qquad \qquad \qquad \qquad \qquad \qquad \qquad \qquad \qquad + \mathcal O\Big(\frac{1}{\nu^2}\Big)
  \end{align*}
  in the limit of large $\nu$. Here, the $\mathcal O()$ is stable under taking derivatives wrt $\xi, \eta$, as well as limits $\xi, \eta \to\infty$.
\end{theorem}
Note that for $\eta=0$, Theorem~\ref{T1} reduces to
\begin{equation}
\label{eq:Ehxi}
\begin{aligned}
    \mathbb E\big( g_\xi(X(t))\big) &  = e^{-\theta(1 - e^{-\xi})} \Big(1 +  \frac 1{\nu\alpha} \int_0^{t\alpha} \Big( \exp\Big({\theta(1 - e^{-s})^2}\Big) \\ & \qquad \qquad \qquad \qquad \qquad - \exp\Big({\theta(1 - e^{-(s + \xi)})(1 - e^{-s})}\Big)\Big)ds\Big) + \mathcal O\Big(\frac{1}{\nu^2}\Big).
\end{aligned}    
\end{equation}
Next, we use the above result in order to obtain several derived quantities. We set
$$\bar m_2(x) := \sum_k (k - m_1(x))^2 x_k.$$ 

\begin{corollary}
  \label{cor}
  Let $X$ be a solution of \eqref{eq:ratchetSDE} and $X(0) = \text{Poi}(\theta)$. Then, as $\nu\to \infty$,  the first moment $m_1(X(t ))$ of the frequency profile $X(t )$ obeys
    \begin{align}
          \mathbb E\big(m_1( X(t))\big) & = \theta +  \frac{1}{2\nu\alpha}  \Big(\exp\big(\theta (1-e^{-\alpha t })^2\big) - 1\Big) + \mathcal  O\Big(\frac{1}{\nu^2}\Big) \label{eq:cor1},
          \intertext{while the expected size of the class carrying $k$ mutations satisfies}
          \mathbb{E}\big(X_k(t )\big) & = e^{-\theta}\frac{\theta^k}{k!} + \frac{1}{\nu \alpha} \int_0^{t \alpha} \exp\Big(-\theta e^{-s}\Big)\Big( \exp\Big(-\theta e^{-s}(1 - e^{-s})\Big)\frac{\theta^k}{k!} \label{eq:cor21}\\ \notag & \qquad \qquad \qquad \qquad \qquad \qquad - \frac{\theta^k(1- e^{-s}(1 - e^{-s}))^k}{k!}\Big) ds + \mathcal O\left(\frac{1}{\nu^2}\right).
          \intertext{In particular,}
          \mathbb E\big(X_0(t )\big) & = e^{-\theta} - \frac{1}{\nu \alpha} \int_0^{t  \alpha} \exp\Big(-\theta e^{-s}\Big)\Big(1 - \exp\Big(-\theta e^{-s}(1 - e^{-s})\Big)\Big) ds + \mathcal O\Big(\frac{1}{\nu^2}\Big). \label{eq:cor22}
          \intertext{The expected centered second moment of the frequency profile satisfies}
        \mathbb E\big(\bar m_2(X(t) )\big) & = \theta - \frac{\theta}{\nu\alpha}e^{-t \alpha}(1 - e^{-t \alpha}) \exp\Big(\theta(1 - e^{-t \alpha})^2 \Big) + \mathcal O\Big(\frac{1}{\nu^2}\Big) \label{eq:cor31}.
        \intertext{Finally, the variance of $m_1(X(t ))$ and its covariance with $X_0(t )$ obey}
        \mathbb V(m_1(X(t) )) & = \frac{1}{\nu\alpha} \int_0^{t \alpha} \theta e^{-2s}\big(1 + \theta(1 - e^{-s})^2 \big) \exp\Big(\theta(1 - e^{- s})^2 \Big) ds + \mathcal O\Big(\frac{1}{\nu^2}\Big), \label{eq:cor32} \\
        \operatorname{Cov} (m_1(X(t) ), & X_0(t )) = \frac{e^{-\theta}}{2\nu\alpha} \Big( \exp\Big( \theta(1 - e^{-t \alpha})^2 \Big) - 2\exp\Big( \theta(1 - e^{-t \alpha}) \Big) + 1 \Big) + \mathcal O\Big(\frac{1}{\nu^2}\Big). \label{eq:cor4}          
        \end{align}
\end{corollary}

\begin{remark}[Scaling]
    \label{rem:scaling}
    Note that $\alpha, \lambda$ and $t$ are fixed in Theorem~\ref{T1} and Corollary~\ref{cor}. Consider in contrast the setting $1/\nu \ll \alpha \ll  1$ as $\nu \to \infty$, with $\theta = \frac \lambda \alpha$ not depending on~$\nu$. Then, we can rescale time in \eqref{eq:ratchetSDE} by $dr := \alpha \, dt$ in \, to obtain
    \begin{align}
        \label{eq:ratchetSDEscaling}
    d X_k = \Big( \Big(\sum_{j=0}^\infty (j-k)X_j\Big) X_k + \theta(X_{k-1} - X_k)\Big)dr + \sum_{\ell \neq k}\sqrt{\frac 1{\nu\alpha} X_k X_\ell} \, dW_{k\ell}, \qquad \tag{$\ast_s$}
    \end{align}
    which satisfies the assumptions of the above results with the triple $(\alpha, \lambda, \nu)$ now replaced by  $(1, \theta, \nu \alpha)$. Applying the results, and scaling time back to $dt = \tfrac{1}{\alpha} dr$, we see that in this scenario, Theorem~\ref{T1} and Corollary~\ref{cor} continue to hold with $t$ replaced by $t/\alpha$ and approximation error $\mathcal O\big( \frac{1}{\nu^2}\big)$ replaced by $\mathcal O\big( \frac{1}{(\nu \alpha)^2}\big)$.
\end{remark}

\begin{remark}[Two sanity checks]
\label{rem:sanity}
As we will see in \eqref{eq:dm1}, 
$$ dm_1 = (\lambda - \alpha \bar m_2) dt + \sqrt{\tfrac 1\nu \bar m_2}dW. $$
In particular, 
\begin{align}
\label{eq:sanity}
\frac{d}{dt }\mathbb E\big(m_1(X(t) )\big)
=
\lambda-\alpha\mathbb E\big(\bar m_2(X(t) )\big).
\end{align}
Moreover, from \eqref{eq:dX0}, we see that
%\begin{align*}
%dX_0 & = (\alpha m_1(X) - \lambda) X_0 dt + \sqrt{\tfrac 1N X_0(1-X_0)}dW,
%\end{align*}
%so in particular
\begin{align}
\label{eq:sanity2}
\frac{d}{dt } \mathbb E\big(X_0(t )\big) = \mathbb E\big((\alpha m_1(X(t) ) - \lambda)X_0(t )\big).
\end{align}
The formulas from the above corollary are consistent with \eqref{eq:sanity} and \eqref{eq:sanity2}, as we will see in Section~\ref{ss:sanity}.
\end{remark}

\begin{remark}[Overview of the proof of Theorem~\ref{T1}]\label{rem:overview}
Proposition \ref{martprob} guarantees that the process $X$ satisfying \eqref{eq:ratchetSDE} solves a well-posed martingale problem for the operator $G = G^{(1)} + G^{(2)}$, where $G^{(1)}$ generates the semigroup $S$ of the  deterministic dynamical system of mutation and selection, and $G^{(2)}$ is the (neutral) Fleming-Viot diffusion operator (with the small coefficient~$1/\nu$).
For the functions $x\mapsto f(x) = g_\xi(x)g_\eta(x)$ that appear in Theorem~\ref{T1},  explicit formulas for $S_tf(x)$ and $G^{(2)}S_tf(x)$ are provided by the results of Section~\ref{S:mart}, thus preparing the ground for an application of Duhamel's formula  (see also Appendix~\ref{app:duhamel})
\begin{equation}\label{Duhamelform}
\mathbb E_x(f(X_\tau)) = S_\tau f(x)+\int_0^\tau \mathbb E_x(G^{(2)}S_{\tau-s}f(X(s)))\, ds. 
\end{equation}
Replacing in~\eqref{Duhamelform}  both $x$ and  $X(s)$ by  $\rm Poi(\theta)$ leads to a fully explicit approximation of $\mathbb E_x(f(X_t))$, but makes it necessary to control the emerging approximation error. This is achieved via martingale methods. More precisely, we will consider the processes 
$$M_f(t) := S_{\tau-t}f(X(t))  - \int_0^t G^{(2)}S_{\tau-s}f(X(s))\, ds, \qquad 0\le t \le \tau.$$
Proposition \ref{T:434appl} will provide conditions under which $M_f$ is a martingale, and the results of Section 4 will guarantee that these conditions are satisfied when choosing $f$ either as $g_\xi g_\eta$ or as $\nu G^{(2)}S_t(g_\xi g_\eta)$. The martingales arising from this second choice will then be employed to bound the above-mentioned approximation error. Figure \ref{fig:proof-structure} illustrates the steps in the proof of Theorem \ref{T1},
with pointers to their ingredients prepared in Sections 3, 4 and A.
\end{remark}

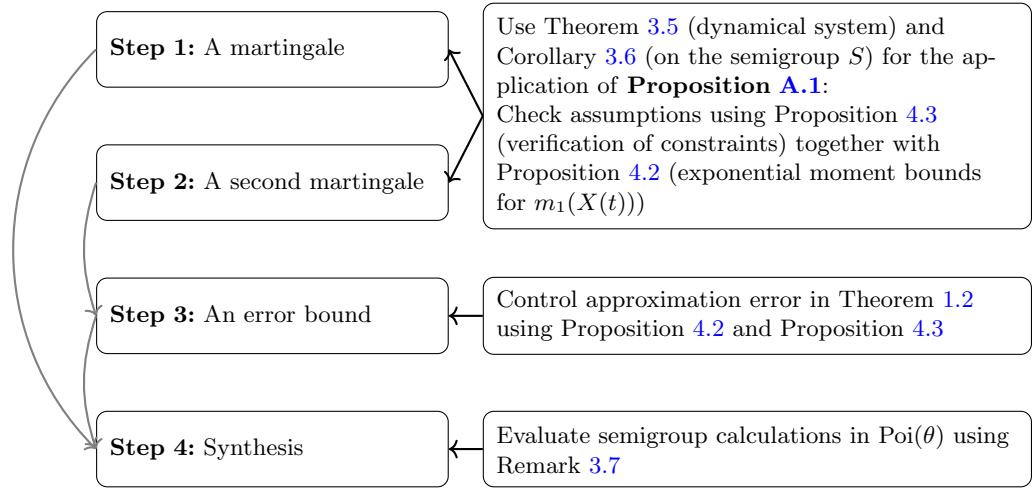
\begin{figure}[H]
  \centering
  \begin{tikzpicture}[
      node distance=0.75cm and 0.45cm,
      box/.style={draw, rectangle, rounded corners, align=left,
                  text width=4.3cm, inner sep=5pt, minimum height=1cm,
                  font=\small},
      tool/.style={draw, rectangle, rounded corners, align=left,
                   text width=6.9cm, inner sep=5pt, minimum height=1cm,
                   font=\small},
      arrow/.style={->, thick},
      dep/.style={->, thick, gray}
  ]

  % Left column: proof steps
  \node[box] (L1) {\textbf{Step 1:} A martingale};
  \node[box, below=of L1] (L2) {\textbf{Step 2:} A second martingale};
  \node[box, below=of L2] (L3) {\textbf{Step 3:} An error bound};
  \node[box, below=of L3] (L4) {\textbf{Step 4:} Synthesis};

  % Right column: tools used in the proof
  \path (L1.east) -- (L2.east) coordinate[midway] (M12);

  \node[tool, right=of M12] (R12) {
   Use Theorem~\ref{T2} (dynamical system) and Corollary~\ref{cor2.4} (on the semigroup $S$) for the application of \textbf{Proposition~\ref{T:434appl}}:\\
   Check assumptions using Proposition~\ref{P:locMartProb}\\ (verification of constraints)
    together with Proposition~\ref{P:bound}
    (exponential moment bounds for $m_1(X(t))$)\\
    %\hspace*{1.2em}$\hookleftarrow$ Lemma~\ref{l:linear}
    %{\small (auxiliary bound)}
  };

  \node[tool, right=of L3] (R3) {
  Control approximation error
  in Theorem~\ref{T1} using Proposition~\ref{P:bound} and
    Proposition~\ref{P:locMartProb}
  };

  \node[tool, right=of L4] (R4) {
    %Equation~\eqref{eq:hxiPoi} (Moment generating function of $\mathrm{Poi}(\theta)$) \\
    Evaluate semigroup calculations in $\text{Poi}(\theta)$ using Remark~\ref{rem:GSPoi}};

  % Tools feed into the steps
  \draw[arrow] (R12.west) -- (L1.east);
  \draw[arrow] (R12.west) -- (L2.east);
  \draw[arrow] (R3.west) -- (L3.east);
  \draw[arrow] (R4.west) -- (L4.east);

  % Dependencies between proof steps
  \draw[dep] (L1.west) to[bend right=45] (L4.west);
  %\draw[dep] (L2.west) to[bend right=30] (L4.west);
  \draw[dep] (L3.west) to[bend right=18] (L4.west);

  % Step 2 feeds into Step 3
  \draw[dep] (L2.west) to[bend right=18] (L3.west);
  %\draw[dep] (L2.south) -- (L3.north);

  \end{tikzpicture}
  \caption{Structure of the proof of Theorem~\ref{T1}.}
  \label{fig:proof-structure}
\end{figure}
\section{Simulations and outlook}
\label{ss:interpretation}

\subsection{Simulations}\label{simul}
\sloppy
In order to assess the scope of our findings, we implement a Wright-Fisher model as follows: For fixed $\mathfrak s \in [0,1)$, $\mathfrak u \geq 0$ and $N\in \mathbb N$, start $X^N(0) \sim \text{Multinomial}(N, \text{Poi}(\mathfrak u/ \mathfrak s))$. In order to generate $X^N(n+1)$ from $X(n)$, set $P(n+1) := \tfrac 1N X(n) \cdot S \cdot M$, where $S = (s_{ij})_{ij}$ with $s_{ij} = (1 - \mathfrak s)^j$ and $M = (m_{ij})_{ij}$ is a matrix with $m_{ij} = \text{Poi}(\mathfrak u)(j-i) = 1_{j \geq i} e^{-\mathfrak u} \mathfrak u^{j-i}/((j-i)!)$. After normalizing $P(n+1)$, pick $X(n+1) \sim \text{Multinomial}(N, P(n+1))$. 

As can be seen from Figure~\ref{fig1}, up to some time, which is related to the mean extinction time of $X_0$ (the ``mean first loss'' of the fittest class), the first order prediction from Corollary~\ref{cor} fits well with the simulations for $\mathbb E(X_0), \mathbb E(X_1), \mathbb E(X_2)$, and $\mathbb E(m_1(X))$. 

\begin{remark}[Model choice]
\begin{enumerate}
    \item Note that, because of the chosen mutation mechanism, there is a positive probability  that an individual gains more than one mutation in one generation.
    \item Let $x = \text{Poi}(\mu)$ for some $\mu \geq 0$. Then, $x\cdot S = \text{Poi}(\mu(1 - \mathfrak s)$ and $x \cdot M \cdot S = \text{Poi}(\mu(1 - \mathfrak s) + \mathfrak u)$. Therefore, for this model, in the limit $N \to \infty$, the equilibrium is $\text{Poi}(\mathfrak u/ \mathfrak s)$, without any assumptions on $\mathfrak s$ and $\mathfrak u$. 
\end{enumerate}   
\end{remark}

\begin{figure}[h!]
  \centering
  \begin{subfigure}[t]{0.48\textwidth}
      \centering
      \includegraphics[width=\linewidth]{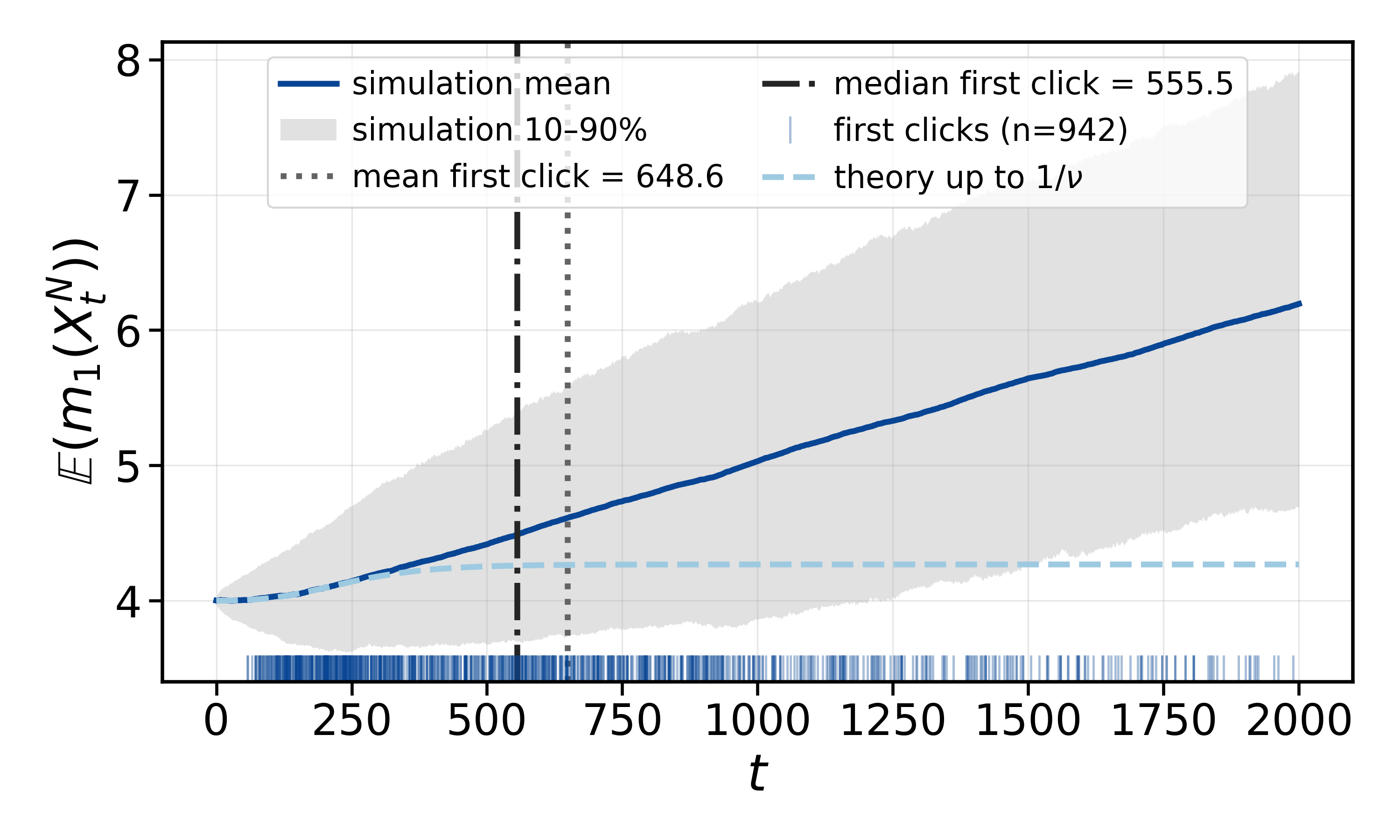}
      \caption{}
      \label{fig:com1A}
  \end{subfigure}
  \hfill
  \begin{subfigure}[t]{0.48\textwidth}
      \centering
      \includegraphics[width=\linewidth]{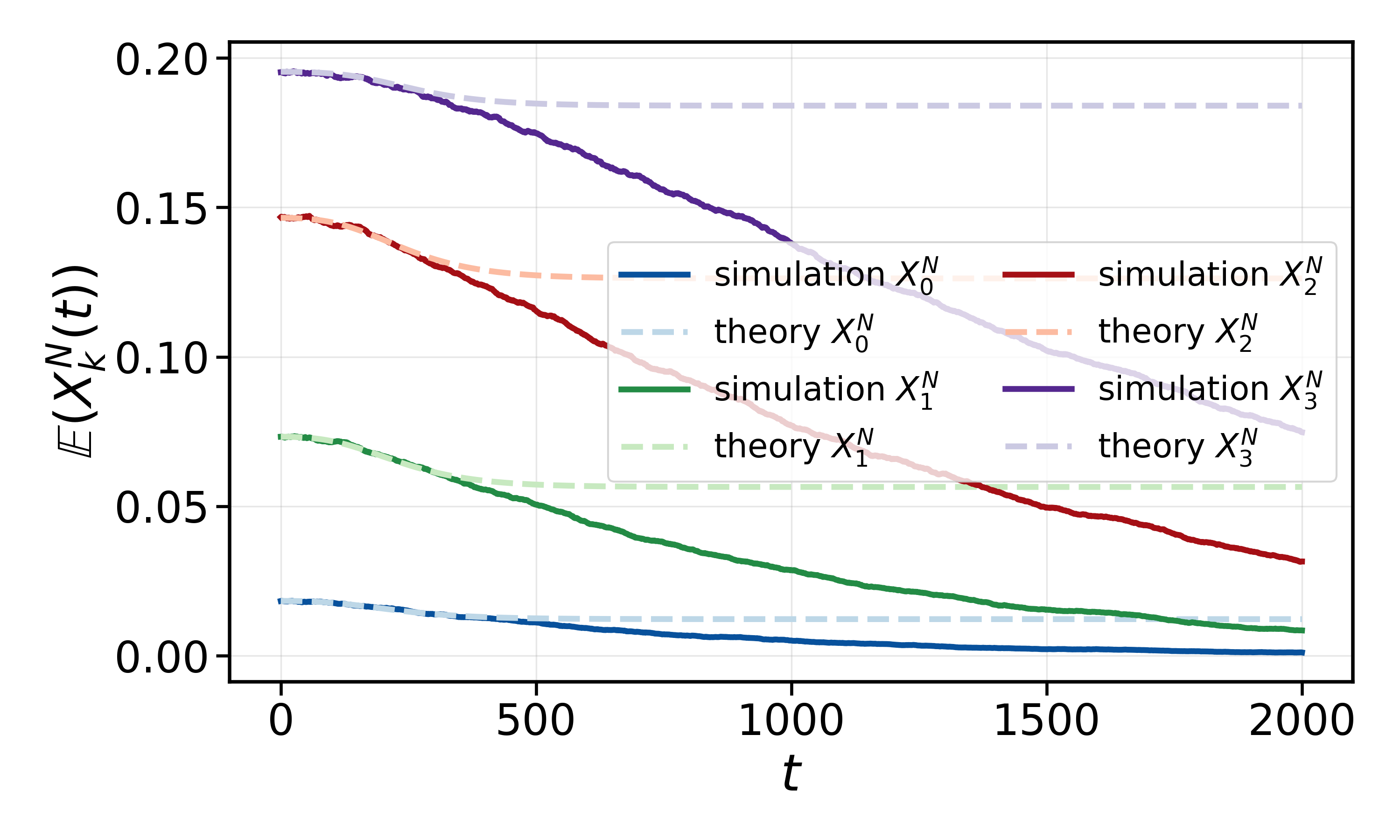}
      \caption{}
      \label{fig:com1B}
  \end{subfigure}
  \caption{\label{fig1}Comparison of the first order predictions from Corollary~\ref{cor} (``theory'') and the empirical values of $\mathbb E(m_1(X(t)))$ (A) and $\mathbb E(X_0(t)), \mathbb E(X_1(t)), \mathbb E(X_2(t))$ (B) for $\mathfrak u = 0.04, \mathfrak s = 0.01$ and $N = 10000$, obtained from a simulation of 1000 paths. Every simulation starts in $X(0)=$  Poi$(\theta)$ and ends at $t=2000$. Up to this time, 942 paths have clicked at least once.}
  \label{fig:com1}
\end{figure}

\subsection{Discussion and outlook} \label{interpret}
As a matter of fact, Corollary~\ref{cor} remains true with $(\alpha, \lambda , \nu)$ in~\eqref{eq:ratchetSDE} replaced by $(1, \theta, \nu \alpha)$ with $\nu \alpha \to \infty$, provided that $\theta$ remains constant; see Remark~\ref{rem:scaling}. Let us now assume that  
\begin{align}
    \label{eq:A}
    \text{\parbox{12cm}{for the parameter triple $(1,\theta, \nu\alpha)$ the assertions of Corollary~\ref{cor} also remain valid for $1\ll \theta = \zeta \log(\nu \alpha)$ with some $\zeta > 0$, and for some times $\tau \gg \log (\theta)$.}}
    \tag{A}
\end{align}
(Concerning this assumption see the discussion in Remark~\ref{rem:notable}.) Like in Remark~\ref{rem:scaling} we can then translate the assertions of Corollary~\ref{cor} back to the solution of~\eqref{eq:ratchetSDE} with $(\alpha, \lambda , \nu)$, by scaling back time using $dt = \tfrac 1\alpha \, dr$. For fixed $c > 0$ we also consider the times $\sigma :=  \tfrac{c}{\lambda} \log\theta$, and observe that
$$\sigma \ll \frac 1\alpha \log \theta \ll \tau.$$
Our assumption on $\tau$ implies that $\theta e^{-\tau \alpha} \downarrow 0$, hence we find from Corollary~\ref{cor} the following (see Section~\ref{rem:interpretation} for the calculations):
\begin{align}
    \label{eq:cor1app}
    \mathbb E\big(m_1(X(\sigma))\big) & \approx \theta + \frac{c^2 \log^2 \theta}{2\nu\lambda}, & \mathbb E\big(m_1(X(\tau))\big) & \approx \theta + \frac{1}{2\nu\alpha e^{-\theta}}, \\
    \label{eq:cor2app}
    \mathbb E\big(X_0(\sigma)\big) & \approx e^{-\theta} \Big(1 - \frac{\theta^{c-1}}{\nu\alpha}\Big), & \mathbb E\big(X_0(\tau)\big) & \approx e^{-\theta} \Big(1 - \frac{\log 2}{\nu\alpha e^{-\theta}}\Big),\\    
    \label{eq:cor14aapp}
    \mathbb E\big(\bar m_2(X(\sigma))\big) & \approx \theta - \frac{c \log\theta}{\nu\alpha}, & \mathbb E\big(\bar m_2(X(\tau))\big) & \approx \theta \Big(1 - \frac{e^{-\tau\alpha}}{\nu\alpha e^{-\theta}} \Big) = \theta + o\Big(\frac{1}{\nu\alpha e^{-\theta}} \Big), \\
    \label{eq:cor3bapp}
    \mathbb V(m_1(X(\sigma))) & \approx \frac{c\log\theta}{\nu\alpha}, & \mathbb V(m_1(X(\tau))) & \approx \frac{1}{4\nu\alpha e^{-\theta}},\\
    \label{eq:cor4app}
    \operatorname{Cov}(m_1(X(\sigma)), & X_0(\sigma)) \approx - \frac{e^{-\theta} \theta^c }{\nu\alpha}, & \operatorname{Cov}(m_1( X(\tau&)), X_0(\tau)) \approx - \frac{1}{2\nu\alpha}.
\end{align}
Note that by time $\sigma$, all quantities are still near their initial value for Poi$(\theta)$ provided that $\nu\alpha \gg \theta^{c-1}$; see \eqref{eq:cor2app}. So, we do not expect any clicks by time $\sigma$. However, let us consider $\psi  := \nu\alpha e^{-\theta}$ and look at times between $\sigma$ and $\tau$, i.e.\ we are considering the time-scale $\tfrac 1\alpha \log\theta$. We have two scenarios:

\begin{itemize}
    \item $\psi \gg 1$: We have $\mathbb E\big(m_1(X(\tau))\big)$, $\mathbb E\big(X_0(\tau)\big)$, $\mathbb E\big(\bar m_2(X(\tau))\big)$ are close to the Poi($\theta$) equilibrium. So, clicks by time $\tau$ are improbable, and in that sense the ratchet clicks rarely in this parameter regime.
    \item $\psi = \mathcal O(1)$: We see from \eqref{eq:cor2app} that $X_0(\tau)$ has (relative to its equilibrium size) already moved away from its deterministic equilibrium $e^{-\theta}$. This indicates that clicks by time $\tau$ have an appreciable probability, and that the ratchet clicks on a time-scale that is of slightly larger order than $\tfrac 1\alpha \log\theta$.
    (Here we assume that $\psi > \log 2$, because otherwise the approximation for $\mathbb E[X_0(\tau)]$ in \eqref{eq:cor2app} becomes meaningless.) 
\end{itemize}

\noindent
In order to see this behavior in simulations of the WF model of Section~\ref{simul}, we proceed as follows (recalling that $(\alpha, \lambda, \nu) \equiv (\mathfrak s, \mathfrak u, N)$ and $\theta = \mathfrak u/\mathfrak s$ in this setting): Fix $\psi > 0$, vary $\delta \in (0,1)$, set $ N\mathfrak s = \psi N^{\delta}$ as well as $\theta = \delta \log N$ (or $N\lambda = \delta \psi N^{\delta} \log N$). Then, we have $N\mathfrak s e^{-\theta} = \psi$, as required. As Figure~\ref{fig2} shows, we see that all average first click times are above the threshold of $\tfrac{\log\theta}{\mathfrak s}$. Moreover, the interpretation from above is well covered. Restricting the choice of simulations to cases where $\theta > 1$ (since we assume that $\theta$ is large) and $\mathfrak s < 0.2$ (such that the diffusion approximation is valid), we find a stark contrast between $\psi \approx 1$ in Figure~\ref{fig2}(A) and $\psi > 1$ in Figure~\ref{fig2}(B), if reporting click times in the time scale $\tfrac{\log\theta}{\mathfrak s}$. This finding is valid for $N\mathfrak s = 10,..., 2000$. Some more information, e.g.\ exact values, as well as $\mathfrak s$ and $\mathfrak u$ for these simulations, are given in Table~\ref{tab1}. Altogether, these simulations confirm the interpretation from above.

\begin{figure}[h!]
  \centering
  \begin{subfigure}[t]{0.48\textwidth}
      \centering
      \includegraphics[width=\linewidth]{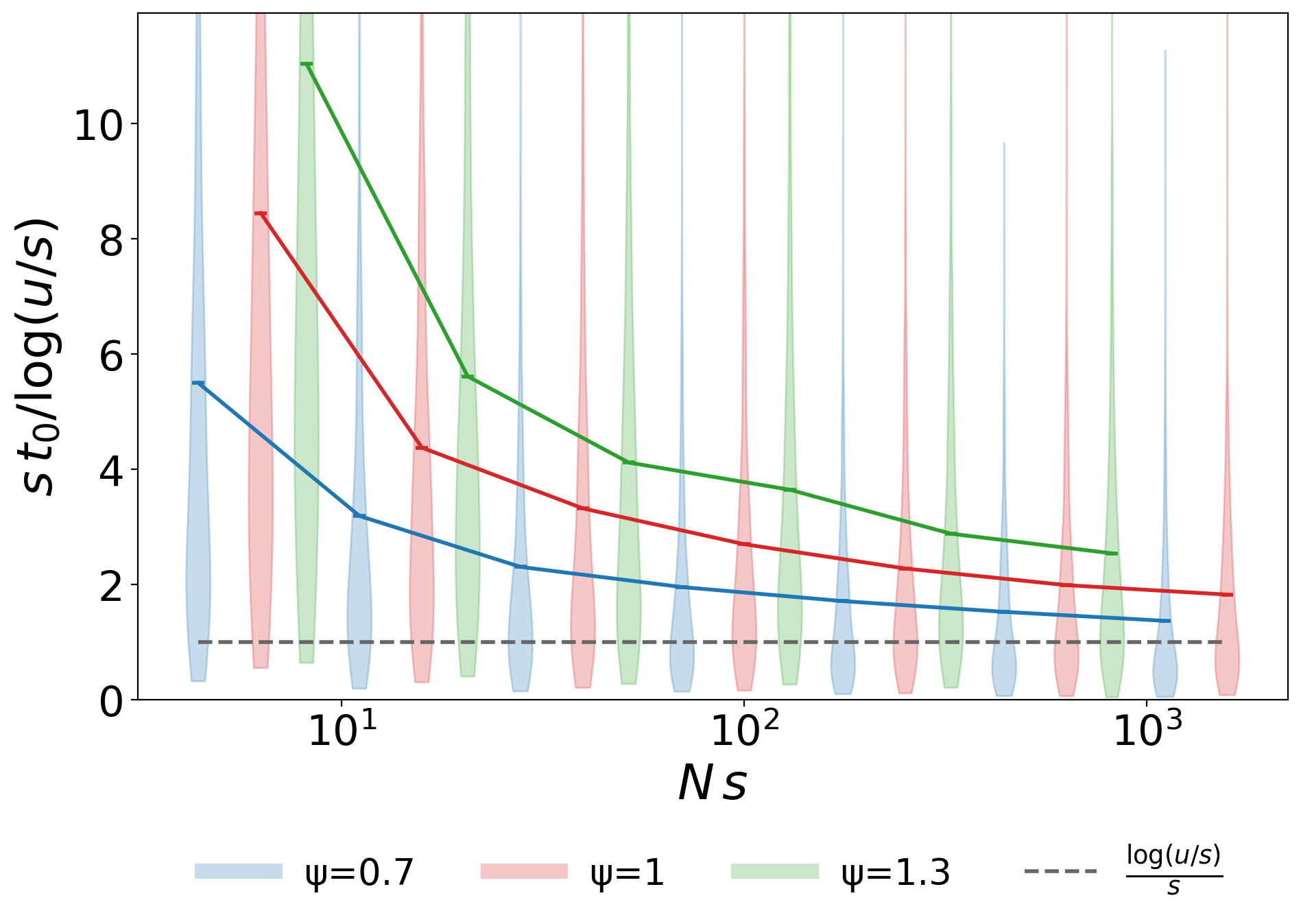}
      \caption{}
      \label{fig:com2A}
  \end{subfigure}
  \hfill
  \begin{subfigure}[t]{0.48\textwidth}
      \centering
      \includegraphics[width=\linewidth]{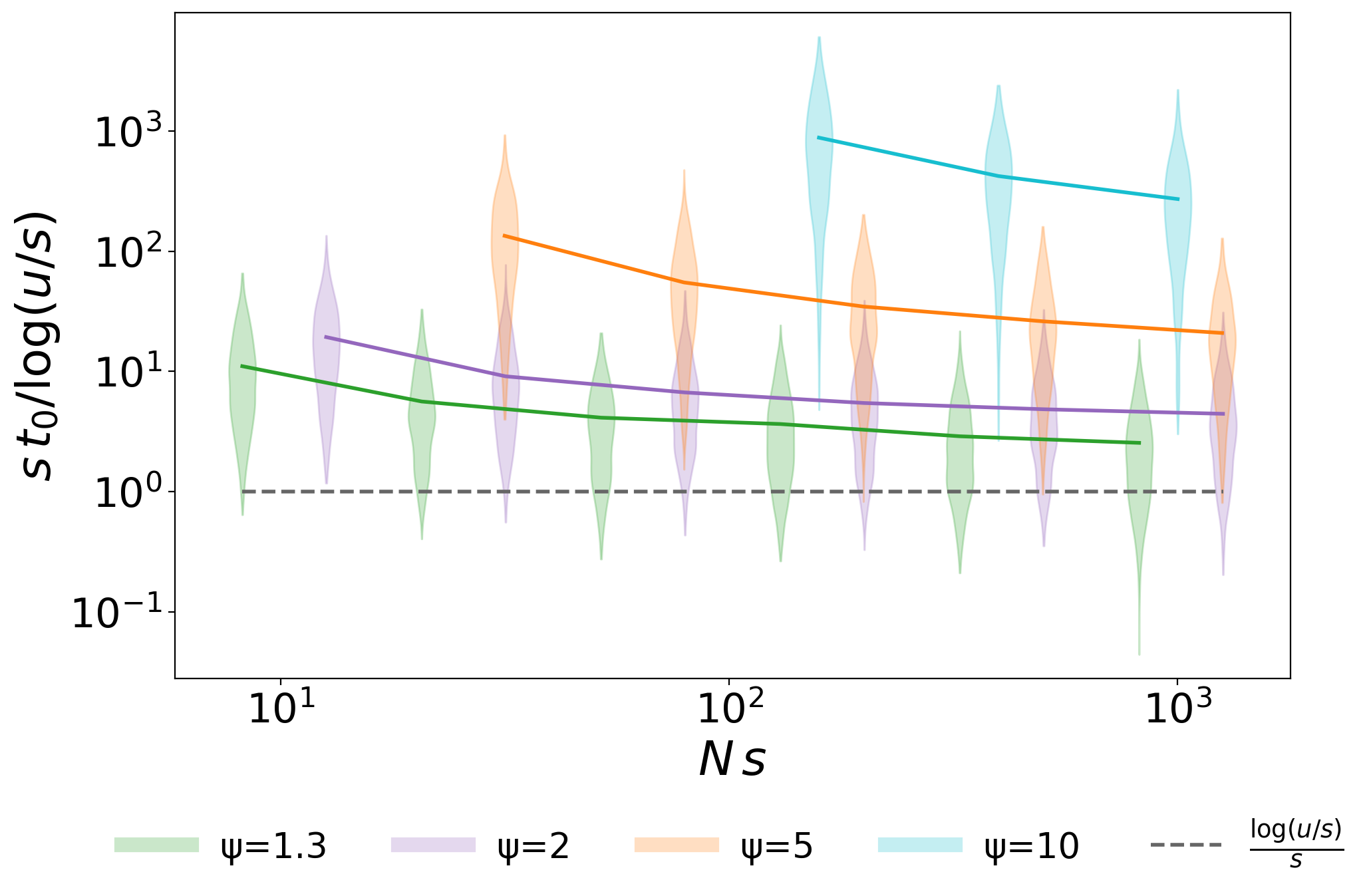}
      \caption{}
      \label{fig:com2B}
  \end{subfigure}
  \caption{\label{fig2}Simulated values of times of first clicks for population size $N=10^4$ in $10^3$ simulations per parameter combination, when starting the system in Poi$(\theta)$. Parameter combinations where not all simulations showed a click by time $10^6$, or $\theta < 1$ or $\mathfrak s > 0.2$ are not shown. (A) For $\psi = \mathcal O(1)$, the time to the first click is larger than $\tfrac{\log\theta}{\mathfrak s}$. (B) For $\psi \gg 1$, clicks are much rarer; note the logarithmic scale on the $y$-axis. Some more information is contained in Table~\ref{tab1} in Section~\ref{simtable}.}
  \label{fig:com2}
\end{figure}

\begin{remark}[On unbounded mutation-selection ratios and time horizons]
\label{rem:notable}
So far we were not able to verify the assumption \eqref{eq:A} stated at the beginning of Section~\ref{interpret}. the main obstacle being that Proposition~\ref{P:bound} in the $(1,\theta,  \nu\alpha )$-scaling  requires a time horizon that must  remain bounded as $\nu \alpha \to \infty$. The integrability of $1/(g_\xi(X_t))^a$ with $a>0$ from \eqref{eq:morebounds} (which is essential for the applicability of Proposition~\ref{T:434appl}) is currently ensured via the somewhat crude inequality~\eqref{eq:jens}, which still leaves hope for the validity of assumption \eqref{eq:A}. 
\end{remark}

\begin{remark}[Outlook]
Two extensions of  the analysis of Muller's ratchet offer themselves, based on the techniques  used and developed in this paper:
\begin{itemize}
   \item Approximations of higher order: For small sampling rate $\tfrac 1 \nu$ we have computed a first order approximation in Muller's ratchet, which is precise up to an approximation error of order $\mathcal O(\tfrac 1 {\nu^2})$.  It would as well be possible -- using iterations of the Duhamel formula~\eqref{Duhamelform} -- to compute higher order approximations. (For this, Step~3 in the proof of Theorem~\ref{T1} would have to be refined.) 
   \item Starting in a different initial state: In Theorem~\ref{T1} and Corollary~\ref{cor} the initial state is $X(0)=\text{Poi}(\theta)$. In principle, this can be relaxed. (Technically, all applications of \eqref{fixedpoint} need to be replaced by applications of Corollary~\ref{cor2.4}.)
\end{itemize}
We postpone these extensions to future research, and focus here on developing a martingale-based perturbation approach for the analysis of Muller's ratchet that will lead to the proof of Theorem~\ref{T1}.
\end{remark}

\section{Foundations}
\label{S:mart}
We study the system \eqref{eq:ratchetSDE} via a well-posed martingale problem, which we introduce in Section~\ref{ss:mp}. Afterwards, in Section~\ref{ss:dyn} we give some concrete calculations for the system with $\nu=\infty$. 
\subsection{Notation and a well-posed martingale problem}
\label{ss:mp}
We first introduce some notation, before we can name a theorem about martingale problems for the ratchet. 

\begin{remark}[Notation]
  For a complete and separable metric space $(E,r)$, we denote by $\mathcal P(E)$ the space of probability measures on (the Borel sets of) $E$. Moreover, $\mathcal M(E)$ ($\mathcal B(E)$) is the space of real-valued, measurable (and bounded) functions. If $E\subseteq \mathbb R^{\mathbb N_0}$, we let $\mathcal C^k(E)$ ($\mathcal C_b^k(E)$) be the (bounded), $k$ times partially continuously differentiable functions (with bounded derivatives). 
\end{remark}

\begin{definition}[Martingale problem]
Let \label{def:mp} $(E,r)$ be a complete and separable metric space, $\mathbb P_0 \in\mathcal P(E)$, $\mathcal F \subseteq \mathcal C(E)$, and $G$ a linear operator on  ${\mathcal C}(E)$ with domain $\mathcal F$. A (distribution $\mathbb P$ of an) $E$-valued stochastic process $X=(X(t))_{t\geq 0}$ is called a solution of the $(E, \mathbb P_0, G,\mathcal F)$ (local) martingale problem if $X_0$ has distribution $\mathbb P_0$, $\mathcal X$ has paths in the space ${\mathcal D}_{E}([0,\infty))$, almost surely, and for all $f\in\mathcal F$,
\begin{equation}\label{13def}
M_f := \Big(f(X(t))-f(X_0) - \int_0^t G f(X(s)) ds \Big)_{t\geq 0}
\end{equation}
is a (local) $\mathbb{P}$-martingale with respect to the canonical filtration. Moreover, the $(E, \mathbb P_0, G, \mathcal F)$ (local) martingale problem is said to be well-posed if there is a unique solution $\mathbb{P}.$
\end{definition}

\noindent
In order to state the martingale problem for Muller's ratchet, we need a state space, the domain of the generator, and the generator itself.

\begin{definition}[Ingredients of the Martingale problem for Muller's ratchet]
  \label{def:F} \mbox{}
    \begin{enumerate}
    \item State space: We write, for some $\chi > 0$,
    \begin{align*}
      \mathbb S&:= \Big\{x \in \mathbb R_+^{\mathbb N_0}: \sum_{k=0}^\infty x_k = 1\Big\}, \\
      \mathbb S_\chi &:=\Big\{x \in \mathbb S: \sum_{k=0}^\infty x_ke^{\chi k} < \infty \Big\}. 
    \end{align*}
    Note that $\mathbb S_{\chi}$ is a suitable state space because there is a metric $r_{\chi}$ (defined by  $r_{\chi}(x, y) := \sum_{k=0}^\infty e^{{\chi} k} |x_k-y_k|)$ that turns the space $\mathbb S_{\chi}$ into a complete and separable metric space (see Remark 3.2 of \cite{pfaffelhuber2012muller}). 
    \item Domain of the generator: For $x\in \mathbb S$ and $\varphi_1,...,\varphi_n \in \mathcal B(\mathbb N_0)$, set
    \begin{align}
        \label{eq:fphi}
        f_{\varphi_1,...,\varphi_n} (x) := \langle x, \varphi_1\rangle \cdots \langle x, \varphi_n\rangle, \qquad \langle x, \varphi\rangle := \sum_k x_k \varphi(k).
    \end{align}
    We define 
    \begin{align}
      \label{eq:121}
      \mathcal F := & \text{algebra generated by } \{f_{\varphi_1,...,\varphi_n}: n\in\mathbb N,
      \varphi_1,...,\varphi_n \in \mathcal B(\mathbb N_0)\} \subseteq
      \mathcal C_b(\mathbb S).
     \end{align}
     \item Generator:
      We define the operator $G$ with domain $\mathcal F$ by
      \begin{align}
        \notag 
        G & = G^{(1)} + G^{(2)},\\
        \label{eq:G1}
        G^{(1)} f(x) & = \sum_{k=0}^\infty \Big(\lambda (x_{k-1} - x_k) + \alpha  \sum_{j=0}^\infty(j-k)x_jx_k \Big)\frac{\partial f(x)}{\partial x_k}, \\
        \label{eq:G2}
        G^{(2)}f(x) & = \frac 1{2\nu} \sum_{k,\ell = 0}^\infty x_k ( \delta_{kl} - x_\ell) \frac{\partial^2 f(x)}{\partial x_k \partial x_\ell}.
      \end{align}
      with $\alpha, \lambda, \nu > 0$.
  \end{enumerate}
\end{definition}

\begin{proposition}  [Well-posed martingale problem, \eqref{eq:ratchetSDE}, and exponential moments]\label{martprob}
    Let $x \in \mathbb S_{\chi}$. Then, the $(\mathbb S_{\chi}, \delta_x, G, \mathcal F)$ local martingale problem is well-posed, and its solution is a process with paths in $\mathcal C_{\mathbb S_{\chi}}([0, \infty))$. This process is the unique weak solution to \eqref{eq:ratchetSDE} with $X(0) = x$. In addition, if $f = f_{\varphi_1,..., \varphi_n}$ is of the form~\eqref{eq:fphi} with $\varphi_i(.) \leq Ce^{\xi .}$ for some $C>0$ and $\xi \in \mathbb R$, then $M_f$ as given in \eqref{13def}, is a martingale. 
\end{proposition}

\subsection{The dynamical system}
\label{ss:dyn}
In this section, we solve the system of ODEs
\begin{align}
  \label{eq:dynsys1}
  \dot x_k & = \lambda(x_{k-1}-x_k) + \alpha \sum_{j=0}^\infty (j - k) x_j x_k,
\end{align}
in other words we give the solution of the first order $G^{(1)}$-martingale problem with $G^{(1)}$ from \eqref{eq:G1}.

\noindent
This is done using some martingale property which is stated in the
next section. 

%We use throughout 
%\begin{align}
%    g_{\alpha \xi}(x) := \sum_{k=0}^\infty e^{-\alpha \xi k} x_k. \label{eq:hxi}
%\end{align}
Recall that for $x = \text{Poi}(\vartheta)$, $\vartheta > 0$,
\begin{align}
    \label{eq:hxiPoi}
    g_{\xi}(\text{Poi}(\vartheta)) & = e^{-\vartheta} \sum_{k=0}^\infty e^{- \xi k } \frac{\vartheta^k}{k!} = \exp\big(-\vartheta ( 1 - e^{- \xi})\big).
\end{align}

\noindent
Here, we give a new proof of the solution of \eqref{eq:dynsys1} (see Theorem~2 in \cite{pfaffelhuber2012muller}, and Proposition 4.1 in \cite{EtheridgePfaffelhuberWakolbinger2007}). 

\begin{theorem}[Solution of \eqref{eq:dynsys1}]\label{T2}
  Started in $x \in \mathbb S$, the system \eqref{eq:dynsys1} has the solution
  \begin{align}\label{eq:sol0}
    x_k(t) & = \frac{1}{g_{\alpha t}(x)} \sum_{i=0}^k x_i e^{-\alpha i t}
      \frac{\theta^{k-i}}{(k-i)!} (1-e^{-\alpha t})^{k-i} \cdot \exp\Big(
    -\theta(1-e^{-\alpha t})\Big).
  \end{align}
\end{theorem}

\begin{proof}
  We use a particle system in our proof. Precisely, let $(K_t)_{t\geq
    0}$ be an $\mathbb N_0 \cup \{\dagger\}$-valued, pure Markov jump process starting in $K_0 \sim x(0)$. Here $, \dagger$ is an absorbing state, and if in state $k \in \mathbb N$, it
  \begin{enumerate}
  \item jumps to $k+1$ at rate $\lambda$,
  \item gets absorbed (i.e.\ jumps to $\dagger$) at rate $k\alpha$.
  \end{enumerate}
  Then, it is straight-forward to see that
  \begin{align}
  \label{eq:sfts}
  x_k(t) := \mathbb P(K_t=k|K_t\neq \dagger)
  \end{align}
  solves \eqref{eq:dynsys1}: We compute
  \begin{align*}
  \mathbb P & (K_{t + dt} = k | K_{t + dt} \neq \dagger) = \frac{\mathbb P(K_{t} = k) (1 - dt (\lambda + \alpha k)) + \mathbb P(K_{t} = k-1) \lambda dt }{ \sum_{j=0}^\infty \mathbb P(K_t = j)(1 - dt \alpha j)} + \mathcal O(dt^2) \\ & = 
  \frac{\mathbb P(K_t = k) + dt \lambda (\mathbb P(K_t = k-1) - \mathbb P(K_t = k)) - dt \alpha k \mathbb P(K_t = k)}{\mathbb P(K_t \neq \dagger) (1 - dt \alpha \sum_{j=0}^\infty j \mathbb P(K_t=j) / \mathbb P(K_t \neq \dagger)} \\ & = 
  \mathbb P(K_{t} = k | K_{t} \neq \dagger)\Big(1 + dt \alpha \sum_{j=0}^\infty j \mathbb P(K_t = j | K_t \neq \dagger)\Big) \\ & \qquad + dt \lambda (\mathbb P(K_t = k-1) - \mathbb P(K_t = k)) - dt \alpha k \mathbb P(K_t = k | K_t \neq \dagger)
  \end{align*}
  Rearranging gives
  \begin{align*}
      \frac{d}{dt} \mathbb P & (K_{t + dt} = k | K_{t + dt} \neq \dagger) = -\lambda\Big(\mathbb P (K_{t} = k | K_{t} \neq \dagger) - \mathbb P (K_{t} = k-1 | K_{t} \neq \dagger)\Big) \\ & \qquad \qquad \qquad \qquad \qquad \qquad \qquad + \alpha \sum_{j = 0} (j - k) \mathbb P(K_t = j | K_t \neq \dagger)\mathbb P(K_t = k | K_t \neq \dagger),
  \end{align*}
  i.e.\ $\mathbb P(K_t = k | K_t \neq \dagger)$ solves \eqref{eq:dynsys1}. 
  
  The right hand side of \eqref{eq:sfts} can be computed directly. When starting with $K_0=i\leq k$, a total of $k-i$ mutations must be gained by time $t$. In addition, every mutation starts an exponential clock at rate $\alpha$ which must be over only after time $t$. Hence, writing '$\sim$' for equality up to constants   (not depending on $k$), 
  \begin{equation}\label{eq:sol1}
    \begin{aligned}
      x_k(t) & \sim \sum_{i=0}^k x_i(0) e^{-\alpha i t} \int_0^t dt_1
      \cdots \int_{t_{k-i-1}}^t dt_{k-i} \lambda^{k-i} e^{-\alpha
        (t-t_1)} \cdots e^{-\alpha (t-t_{k - i})} \\ & = \sum_{i=0}^k
      x_i(0) e^{-\alpha i t} \frac{\lambda^{k-i}}{(k-i)!}
      \Big(\int_0^t e^{-\alpha(t-s)}ds\Big)^{k-i} \\ & = \sum_{i=0}^k
      x_i(0) e^{-\alpha i t} \frac{\theta^{k-i}}{(k-i)!} (1-e^{-\alpha
        t})^{k-i}.
    \end{aligned}
  \end{equation}
  Since 
  \begin{align*}
    \sum_{k=0}^\infty & \sum_{i=0}^k x_i e^{-\alpha i t}
    \frac{\theta^{k-i}}{(k-i)!} (1-e^{-\alpha t})^{k-i} =
    \sum_{i=0}^\infty x_i e^{-\alpha i t} \sum_{k=0}^\infty
    \frac{\theta^{k}}{k!} (1-e^{-\alpha t})^{k} \\ & = \exp\Big(
    \theta(1-e^{-\alpha t})\Big) \cdot g_{\alpha t}(x),
  \end{align*}
  normalizing \eqref{eq:sol1} gives \eqref{eq:sol0}. 
\end{proof}
 
\begin{corollary}\label{cor2.4}
  Let $(S_t)_{t\geq 0}$ be the semigroup generated by $G^{(1)}$. Started in $x \in \mathbb S$, the solution of \eqref{eq:dynsys1} satisfies, for $\xi \in \mathbb R_+$, 
\begin{equation}
\label{eq:sol2}
\begin{aligned}
  S_t g_{\xi}(x) & = g_{\xi}(x(t)) = \frac{g_{\alpha t + \xi}(x)}{g_{\alpha t}(x)}
  \exp\Big(-\theta(1-e^{- \xi})(1-e^{-\alpha t})\Big)
\end{aligned}
\end{equation}
as well as, for $\xi, \eta \in \mathbb R_+$,
\begin{align}
    \label{eq:sol3}
    G^{(1)}S_t g_{\xi} & = \exp\Big(-\theta(1-e^{- \xi})(1-e^{-\alpha t})\Big)\Big(\lambda e^{-\alpha t}(1 - e^{-\xi})\frac{g_{\alpha  t + \xi}}{g_{\alpha t}} + \frac{\partial}{\partial t} \frac{g_{\alpha  t + \xi}}{g_{\alpha t}}\Big), \\ \label{eq:sol32}
    G^{(1)}S_t g_{\xi} g_{\eta} & = (S_t g_{\xi})\, G^{(1)}S_t g_{\eta} + (S_t g_{ \eta})\, G^{(1)}S_t g_{ \xi}, \\ 
    \label{eq:sol4}
    \nu \cdot G^{(2)}S_t g_{ \xi} & = 
    \frac{g_{\alpha  t + \xi}(x)g_{2\alpha t} - g_{\alpha t} g_{2\alpha t+\xi}}{g_{\alpha t}^3}\exp\Big(-\theta(1-e^{-\xi})(1-e^{-\alpha t})\Big), \\
    \nu \cdot G^{(2)}S_t g_{\xi} g_{\eta}  & = \exp\Big(-\theta(2-e^{-\xi} - e^{- \eta})(1-e^{-\alpha t})\Big) \notag \\ 
    \label{eq:sol42} &  \cdot 
    \frac{
    3g_{\alpha t + \xi} g_{\alpha  t + \eta}g_{2\alpha t} - 2 g_{\alpha  t + \xi} g_{\alpha t}g_{2\alpha t + \eta} - 2 g_{\alpha t + \eta} g_{\alpha t} g_{2\alpha t+\xi} + g_{\alpha t}^2 g_{2\alpha t+\xi+\eta}}{g_{\alpha t}^4}. 
\end{align}
Moreover, for $f_{t,\xi,\eta} := \nu \cdot G^{(2)}S_t g_{\xi} g_{\eta}$, there are bounded functions $g_{s, t, \xi, \eta}, g^1_{s, t, \xi, \eta}, g^2_{s, t, \xi, \eta} \in \text{span} \{g_{\xi_1} \cdots g_{\xi_m}: m \in \mathbb N, \xi_1,...,\xi_m \in \mathbb R_+ \}$ and $a, a_1, a_2$ (all not depending on $\nu$) such that 
\begin{align}
    \label{eq:morebounds}
    S_s f_{t,\xi, \eta} = \frac{g_{s,t,\xi, \eta}}{g_{\alpha (s+t)}^{a}}, \qquad G^{(1)} S_s f_{t,\xi, \eta} = \frac{g^1_{s,t,\xi, \eta}}{g_{\alpha (s+t)}^{a_1}}, \qquad \nu \cdot G^{(2)} S_s f_{t,\xi, \eta} = \frac{g^2_{s,t,\xi, \eta}}{g_{\alpha (s+t)}^{a_2}}.
\end{align}
\end{corollary}

\begin{remark}\label{rem:GSPoi}
    If $x = \text{Poi}(\theta)$, note that 
    \begin{align*}
        \frac{g_{\alpha s}(x)}{g_{\alpha t}(x)} & = \frac{\exp\Big(-\theta(1 - e^{-\alpha s}) \Big)}{\exp\Big(-\theta(1 - e^{-\alpha t}) \Big)} = \frac{e^{\theta e^{-\alpha s}}}{e^{\theta e^{-\alpha t}}}
    \end{align*}
    by \eqref{eq:hxiPoi}. We additionally calculate
    \begin{align*}
    &\frac{\exp\Big(-\theta(2-e^{- \xi}-e^{-\eta})(1-e^{-\alpha t})\Big)
    }{g_{\xi} g_{\eta}}
    = \frac{g_{\alpha t}^2}{g_{\alpha t +\xi}g_{\alpha t+\eta}}.
\end{align*}
    Therefore and with help of \eqref{eq:sol42}, we get
    \begin{equation}
    \begin{aligned}
    \label{eq:sol42Poi}
    \nu \cdot & \frac{G^{(2)}S_t g_{\xi} g_{\eta}(x)}{g_{\xi}(x) g_{ \eta}(x)} = \exp\Big(-\theta(2-e^{- \xi} - e^{- \eta})(1-e^{-\alpha t})\Big)  \\ 
   & \qquad \qquad \cdot %\frac 1\nu 
    \frac{
    3g_{\alpha  t + \xi} g_{\alpha t + \eta}g_{2\alpha t} - 2 g_{\alpha t + \xi} g_{\alpha t}g_{2\alpha t + \eta} - 2 g_{\alpha t + \eta} g_{\alpha t} g_{2\alpha t+\xi} + g_{\alpha t}^2 g_{2\alpha t+\xi+\eta}}{g_{\alpha t}^4 g_{ \xi}(x) g_{ \eta}(x)} \\
    &=  \frac{3g_{2\alpha t} - 2 g_{\alpha t}g_{2\alpha t + \eta} / g_{\alpha t + \eta} - 2 g_{\alpha t} g_{2\alpha t+\xi} / g_{\alpha t + \xi} + g_{\alpha t}^2 g_{2\alpha t+\xi+\eta}/ (g_{\alpha t + \xi} g_{\alpha t + \eta}) }{g_{\alpha t}^2(x)} \\ 
    & = \frac{1}{g_{\alpha t}(x)}\Big( 3 \exp\Big( -\theta e^{-\alpha t}(1 - e^{-\alpha t})\Big) \\ & \qquad - 2 \exp\Big( - \theta e^{-(\alpha t + \eta)}(1 - e^{-\alpha  t})\Big) - 2\exp\Big(  - \theta e^{-(\alpha t + \xi)}(1 - e^{-\alpha t}) \Big) \\ &  \qquad \qquad \qquad + \exp\Big( -\theta(e^{-(\alpha t + \xi)} + e^{-(\alpha t + \eta)} -e^{-\alpha t} - e^{-(2\alpha t + \xi + \eta)} ) \Big) \Big) \\ & = 3 \exp\Big(\theta(1 - e^{-\alpha t})^2 \Big) - 2 \exp\Big(\theta(1 - e^{-(\alpha t + \xi)})(1 - e^{-\alpha t}) \Big) \\ & \qquad - 2 \exp\Big(\theta(1 - e^{-(\alpha t + \eta)})(1 - e^{-\alpha t}) \Big) + \exp\Big( \theta(1 - e^{-(\alpha t + \xi)})(1 - e^{-(\alpha t + \eta)})\Big).
    \end{aligned}
    \end{equation}
    In particular, setting $\eta=0$,
    \begin{equation}
    \begin{aligned}
    \nu \cdot & \frac{G^{(2)}S_t g_{ \xi} }{g_{ \xi}(x)} = \exp\Big(\theta(1 - e^{-\alpha t})^2 \Big) -  \exp\Big(\theta(1 - e^{-(\alpha t + \xi)})(1 - e^{-\alpha t}) \Big).
    \end{aligned}
    \end{equation}
\end{remark}

\begin{remark}
    Note that for $x = \text{Poi}(\varphi)$, using \eqref{eq:hxiPoi}  and \eqref{eq:sol2},
    \begin{align*}
        S_tg_{ \xi}(x) & = \exp\Big(-\varphi(1 - e^{-(\alpha t + \xi)}) + \varphi(1 - e^{-\alpha t})) \Big) \exp\Big(-\theta(1-e^{-\xi})(1-e^{-\alpha t})\Big) \\ & = \exp\Big( - \theta(1 - e^{-\xi}) + ( \theta - \varphi)(1 - e^{-\xi})e^{-\alpha t}  \Big) \\ & = \exp\Big( - (\theta + (\varphi - \theta) e^{-\alpha t}) (1 - e^{-\xi})\Big).
    \end{align*}
    This shows that $x(t) = \text{Poi}((\theta + (\varphi - \theta) e^{-\alpha t}))$, i.e.\ the system stays Poisson, and the parameter evolves according to $\dot \varphi = \alpha (\theta - \varphi)$.
\end{remark}

\begin{proof}[Proof of Corollary~\ref{cor2.4}]
  For \eqref{eq:sol2}, we compute
  \begin{align*}
    g_{ \xi}(x(t)) & = \sum_{k=0}^\infty e^{- \xi k} x_k(t) \\ & = \frac{1}{g_{\alpha t}(x)} \! \sum_{i=0}^\infty x_i e^{-i(\alpha t+\xi)} \Big(\underbrace{ \sum_{k=i}^\infty  \frac{\theta^{k-i} e^{- (k-i)\xi}}{(k-i)!}
      (1-e^{-\alpha t})^{k-i}}_{= \exp\big( \theta(e^{- \xi} ( 1\! - e^{-\alpha t}) \big) }\Big)\Big)\exp\Big(\!-\theta(1 \! -e^{-\alpha t})\Big)\\ & =
    \frac{g_{\alpha t + \xi}(x)}{ g_{\alpha t}(x)}
    \exp\Big(-\theta(1-e^{- \xi})(1-e^{-\alpha t})\Big).
  \end{align*}
  Then,
  \begin{align*}
    G^{(1)}S_t g_{\xi}(x) & = \frac{\partial}{\partial t}S_t g_{\xi}(x) \\ & = \exp\Big(-\theta(1-e^{- \xi})(1-e^{-\alpha t})\Big)\Big(\lambda e^{-\alpha t}(1 - e^{- \xi})\frac{g_{\alpha  t + \xi}(x)}{g_{\alpha t}(x)} + \frac{\partial}{\partial t} \frac{g_{\alpha t + \xi}(x)}{g_{\alpha t}(x)}\Big),
\end{align*}
which shows \eqref{eq:sol3}. From this, \eqref{eq:sol32} follows since $S_t(g_{\xi} g_{\eta}) = (S_t g_{\xi})(S_tg_{\eta})$ and $G^{(1)}$ is first order. 

Next, for $G^{(2)}S_t g_{\xi}$, taking a look at the form of $S_tg_{\xi}(x)$, we need to compute
\begin{align*}
  \sum_{k,\ell=0}^\infty & 
  x_k(\delta_{k\ell} - x_\ell) \frac{\partial^2}{\partial x_k
    \partial x_\ell} \frac{g_{\alpha t + \xi}(x)}{g_{\alpha t}(x)} 
  \\
  & = \sum_{k,\ell=0}^\infty
  x_k(\delta_{k\ell} - x_\ell) \frac{\partial}{\partial x_k} \frac{g_{\alpha t}(x) e^{-(\alpha t + \xi)\ell} - g_{\alpha t + \xi}(x) e^{-\alpha t \ell}}{g_{\alpha t}(x)^2} 
  \\ & = \sum_{k,\ell=0}^\infty
  x_k(\delta_{k\ell} - x_\ell) \cdot \big(g_{\alpha t}(x) ( e^{-\alpha t k} e^{-(\alpha t + \xi)\ell} - e^{-(\alpha t + \xi) k} e^{-\alpha t \ell}) \\ & \qquad \qquad \qquad \qquad \qquad - 2 e^{-\alpha tk}(g_{\alpha t}(x) e^{-(\alpha t + \xi)\ell} - g_{\alpha t + \xi}(x) e^{-\alpha t \ell})\big) / g_{\alpha t}(x)^3
  \\ & =  
  -2\frac{g_{\alpha t}(x)g_{2\alpha t+\xi}(x) - g_{\alpha t +\xi}(x)g_{2\alpha t}(x)}{g_{\alpha t}(x)^3}.
\end{align*}
Therefore, 
\begin{align*}
    \nu G^{(2)} S_{t} g_{ \xi}(x) & = \frac{1}{2} \sum_{k,\ell = 0}^\infty x_k ( \delta_{kl} - x_\ell) \frac{\partial^2 S_{t} g_{ \xi} (x)}{\partial x_k \partial x_\ell} \\ & = 
    \frac{1}{2} \sum_{k,\ell = 0}^\infty x_k ( \delta_{kl} - x_\ell) \frac{\partial^2 }{\partial x_k \partial x_\ell} \frac{g_{\alpha t +\xi}(x)}{g_{\alpha t}(x)}
    \exp\Big(-\theta(1-e^{-\xi})(1-e^{-\alpha t})\Big) \\ & = 
    \frac{g_{\alpha t + \xi}(x)g_{\alpha 2t}(x) - g_{\alpha t}(x) g_{2\alpha t+\xi}(x)}{g_{\alpha t}(x)^3}\exp\Big(-\theta(1-e^{- \xi})(1-e^{-\alpha t})\Big).
\end{align*}
Let us turn to \eqref{eq:sol42}. We calculate, omitting the $x$-dependence,
\begin{align*}
  \sum_{k,\ell=0}^\infty & 
  x_k(\delta_{k\ell} - x_\ell) \Big(\frac{\partial}{\partial x_k}\frac{g_{\alpha  t + \xi}}{g_{\alpha t}} \Big) \Big(\frac{\partial}{\partial x_\ell} \frac{g_{\alpha t + \eta}}{g_{\alpha t}}\Big)
  \\
  & = \sum_{k,\ell=0}^\infty
  x_k(\delta_{k\ell} - x_\ell) \frac{(g_{\alpha t} e^{-(\alpha t + \xi)k} - g_{\alpha t + \xi} e^{-\alpha t k})(g_{\alpha t} e^{-(\alpha t + \eta)\ell} - g_{\alpha t + \eta} e^{-\alpha t \ell})}{g_{\alpha t}^4} 
  \\ & = \Big(g_{\alpha t}^2 g_{2\alpha t+\xi+\eta} - g_{\alpha t} g_{2\alpha t+\xi} g_{\alpha t + \eta} - g_{\alpha t} g_{2\alpha t+\eta}g_{\alpha t + \xi} + g_{\alpha t +\xi}g_{\alpha t + \eta}g_{2\alpha t})\Big) / g_{\alpha t}^4.
\end{align*}
Therefore, since $S_t g_{ \xi} g_{ \eta} = (S_t g_{ \xi})(S_t g_{ \eta})$, \eqref{eq:sol42} follows from
\begin{align*}
    \exp\Big( & \theta(2 - e^{- \xi} - e^{- \eta})(1 - e^{-\alpha t}) \Big) \nu G^{(2)}S_t g_{ \xi} g_{ \eta}  \\ & = \nu\frac{g_{\alpha t + \xi}}{g_{\alpha t}} G^{(2)} \frac{g_{\alpha t + \eta}}{g_{\alpha t}} + \nu\frac{g_{\alpha t + \eta}}{g_{\alpha t}} G^{(2)} \frac{g_{\alpha t + \xi}}{g_{\alpha t}} \\ & \qquad \qquad \qquad \qquad + \sum_{k,\ell=0}^\infty 
    x_k(\delta_{k\ell} - x_\ell) \Big(\frac{\partial}{\partial x_k}\frac{g_{\alpha t + \xi}}{g_{\alpha t}} \Big) \Big(\frac{\partial}{\partial x_\ell} \frac{g_{\alpha t + \eta}}{g_{\alpha t}}\Big)
    \\ & = 
    \Big(g_{\alpha t + \xi}(g_{\alpha t + \eta}g_{\alpha 2t} - g_{\alpha t}g_{2\alpha t + \eta}) + g_{\alpha t + \eta} (g_{\alpha t + \xi}g_{2\alpha t} - g_{\alpha t}g_{2\alpha t+\xi}) \\ & \qquad + g_{\alpha t}^2 g_{2\alpha t+\xi+\eta} - g_{\alpha t} g_{2\alpha t+\xi}g_{\alpha t + \eta} - g_{\alpha t} g_{2\alpha t+\eta}g_{\alpha t + \xi} + g_{\alpha t +\xi}g_{\alpha t + \eta}g_{2\alpha t}\Big)/g_{\alpha t}^4 \\ & = \Big(
    2g_{\alpha t + \xi}(g_{\alpha t + \eta}g_{2\alpha t} - g_{\alpha t}g_{2\alpha t + \eta}) + 2g_{\alpha t + \eta}(g_{\alpha t + \xi}g_{2\alpha t} - g_{\alpha t}g_{2\alpha t+\xi}) \\ & \qquad \qquad \qquad \qquad \qquad \qquad \qquad \qquad + g_{\alpha t}^2 g_{2\alpha t+\xi+\eta} - g_{\alpha t + \xi}g_{\alpha t + \eta}g_{2\alpha t}\Big) / g_{\alpha t}^4.
\end{align*}
Let us prove \eqref{eq:morebounds}. Define
$$
c(s,\xi):=
\exp\Big(
-\theta(1-e^{-\xi})(1-e^{-\alpha s})
\Big).
$$
Then, it holds
$$
S_s g_{\xi}  =c(s,\xi)\frac{g_{\alpha s+\xi}  }{g_{\alpha s}  },
$$
which leads to
\begin{align*}
S_s\left(
\frac{
g_{\xi}g_{\eta}g_{\zeta}
}{g_{\alpha t}^4}
\right)  
&=
\frac{
S_s g_{\xi}  S_s g_{\eta}  S_s g_{\zeta}  
}{
(S_s g_{\alpha t}  )^4
} \\
&=
\frac{
c(s,\xi)c(s,\eta)c(s,\zeta)
}{
c(s,\alpha t)^4
}
\frac{
\frac{g_{\alpha s + \xi}  }{g_{\alpha s}  }
\frac{g_{\alpha s + \eta}  }{g_{\alpha s}  }
\frac{g_{\alpha s + \zeta}  }{g_{\alpha s}  }
}{
\left(\frac{g_{\alpha (s+t)}  }{g_{\alpha s}  }\right)^4
} \\
&=
\frac{
c(s,\xi)c(s,\eta)c(s,\zeta)
}{
c(s,\alpha t)^4
}
\frac{
g_{\alpha s}  g_{\alpha s + \xi}  g_{\alpha s + \eta}  g_{\alpha s + \zeta}  
}{
g_{\alpha (s+t)}  ^4
}.
\end{align*}

The rest of the proof of \eqref{eq:morebounds} follows along the same lines.
\end{proof}

\section{Some properties of the process $X$}
\label{S:propX}

\subsection{Moment bounds}
\label{ss:mombounds}
In this section, we establish in Proposition~\ref{P:bound} some bound needed for the application of Proposition~\ref{T:434appl}. More precisely, note that by Jensen's inequality, $e^{- \xi m_1(x)} \leq g_{\xi}(x)$, and hence
\begin{align}
    \label{eq:jens}
    \frac{1}{g_{ \xi}(x)^a} \leq e^{a \xi m_1(x)}, \qquad x \in \mathbb S, a \geq 0.
\end{align} In the proof of the main theorem, we need to show an integrability condition for $g_\xi(X(t))^{-a}$; for this we will rely on~\eqref{eq:jens}. 
Therefore, we need to bound the tails of $\sup_{0 \leq t \leq \tau} m_1(X(t))$.
For this, we recall Lemma 3.2 from \cite{audiffren2013muller}, which we give here in our notation for completeness.

\begin{lemma}\label{l:linear}
   Let \eqref{eq:ratchetSDE} have a unique weak solution and $\tau\geq 0$. For all $c>0$, 
  $$\mathbb P\left(\sup_{0 \leq t \leq \tau} m_1(X(t)) - m_1(X_0) \geq \lambda \tau + c\right) \leq \exp(-2\alpha \nu c).$$
\end{lemma}

\begin{proof}
Using It\^o's Lemma in \eqref{eq:ratchetSDE}, we see that
\begin{equation}
\label{eq:dm1}
\begin{aligned}
   dm_1(x) & = \sum_{k=0}^\infty k dX_k \\ & = \alpha \sum_{j,k=0}^\infty k(j-k) X_jX_k + \lambda  \sum_{k=0}^\infty k( X_{k-1} -  X_k)\Big)dt + \sum_{k,\ell \atop k \neq \ell} k \sqrt{\frac 1\nu X_k X_\ell} dW_{k\ell} \\ & = (\lambda -\alpha \bar m_2(X)) dt + \sqrt{\frac 1\nu \bar m_2(X)}dW
\end{aligned}    
\end{equation}
for some Brownian motion $W$, since 
\begin{align*}
    \Big\langle \sum_{k, \ell \atop k \neq \ell} k \! \sqrt{\frac 1\nu X_k X_\ell} dW_{k\ell}, \sum_{k', \ell' \atop k' \neq \ell'} k' \sqrt{\frac 1\nu X_{k'} X_{\ell'}} dW_{k'\ell'}  \Big\rangle \!=\! \frac 1\nu \sum_{k,\ell \atop k \neq \ell} k^2 X_k X_\ell - k\ell X_k X_\ell \!=\! \frac 1\nu \bar m_2(X).
\end{align*}
So, we find 
\begin{align*}
m_1(X(t)) - m_1(X_0) = \int_0^t dm_1(X(s)) = \lambda t \underbrace{- \alpha \int_0^t \bar m_2(X(s))ds + \int_{0}^{t} \sqrt{\frac{\bar m_2(X(s))}{\nu}}\, dW_s}_{=:Z_t}.
\end{align*}
Moreover, (see e.g.~\cite{Kallenberg20021},a 19.22, Exercise 2 in Chapter 18)
\[
\big(\exp\bigl(2\alpha \nu\, Z_t\bigr)\big)_{t\geq 0} = 
\Big(\exp\Bigl( \int_0^{t} 2\alpha \sqrt{\nu \bar m_2(X(s))} dW_s - \int_0^{t} 2\alpha^2 \nu \bar m_2(X(s)) ds\Big)\Big)_{t \geq 0}
\]
is both a local martingale and a supermartingale. 
In particular, using Doob's inequality, for any $c>0$,
\begin{align*}
\mathbb P\Big( & \sup_{0 \le t \le \tau} (m_1(X(t)) - m_1(X_0)) > \lambda \tau + c\Big) = \mathbb P\Big( \sup_{0 \le t \le \tau} \lambda t + Z_t > \lambda \tau + c\Big) \\ & \leq \mathbb P\Big( \sup_{0 \le t \le \tau} Z_t > c\Big) = \mathbb P\!\left( \sup_{0 \le t \le \tau} \exp\Big( 2\alpha \nu Z_t \Big) \ge e^{2\alpha \nu c} \right) \leq e^{-2\alpha \nu c}.
\end{align*}
\end{proof}

\begin{proposition}
\label{P:bound}
Let $m_1(X_0) = \zeta$ and $\eta < 2\alpha \nu$. Then, 
$$ \mathbb E\Big( \sup_{0 \leq t\leq \tau} \exp\Big(\eta  m_1(X(t))\Big) \Big) < \infty$$
\end{proposition}

\begin{proof}
    There is nothing to show for $\eta \leq 0$, so we assume that $0 < \eta < 2\alpha \nu$. We write, using Lemma~\ref{l:linear},
\begin{align*}
    \mathbb E\Big(&e^{- \eta(\zeta + \lambda\tau)} \sup_{0 \leq t\leq \tau} \exp\Big(\eta  m_1(X(t))\Big) \Big) \\ & = \int_0^\infty \mathbb P\Big(e^{- \eta(\zeta + \lambda\tau)}\exp\Big(\eta  \sup_{0 \leq t\leq \tau} m_1(X(t))\Big) \geq x \Big) dx \\ 
     & = \int_1^\infty \mathbb P\Big(e^{- \eta(\zeta + \lambda\tau)}\exp\Big(\eta  \sup_{0 \leq t\leq \tau} m_1(X(t))\Big) \geq x \Big) dx \\ 
     &\quad \quad +\int_0^1 \mathbb P\Big(e^{- \eta(\zeta + \lambda\tau)}\exp\Big(\eta  \sup_{0 \leq t\leq \tau} m_1(X(t))\Big) \geq x \Big) dx\\
     & \leq 1 + \int_1^\infty \mathbb P\big( \sup_{0 \leq t\leq \tau} m_1(X(t)) \geq \zeta + \lambda \tau + \tfrac{\log x}{\eta} \big) dx \\ 
     & \leq 1 + \int_1^\infty \exp\big( - 2\alpha \nu \tfrac{\log x}{\eta} \big)dx\\
     &= 1+ \int_1^\infty x^{-2\alpha \nu / \eta} dx < \infty. 
\end{align*}
\end{proof}

\subsection{A local martingale problem}
\label{S:locMartProb}
We will establish that $X$ from Proposition~\ref{martprob} not only solves that $(\mathbb S_{\chi}, \delta_x, G, \mathcal F)$ martingale problem, but also a local martingale problem for some $\mathcal D$ containing unbounded functions. 

\begin{proposition}
    \label{P:locMartProb}
    Recall all notions from Definition~\ref{def:F}, and $S$, the semigroup generated by $G^{(1)}$, from Corollary~\ref{cor2.4}. For $x \in \mathbb S_{\chi}$, let $X$ be the solution of the $(\mathbb S_{\chi}, \delta_x, G, \mathcal F)$ martingale problem from Proposition~\ref{martprob}. In addition, let 
    \begin{align}\label{def:D}
    \mathcal D & := \text{algebra generated by } \bigcup_{{\zeta} \geq 0} \bigcup_{m=0}^\infty \mathcal D^m_{\zeta}, \\ 
    \mathcal D^m_{\zeta} & := \left\{ f_{\xi{, \zeta}} := \frac{h_{ \xi_1} \cdots h_{ \xi_m}}{g_{\alpha {\zeta}}^m}
    : h_{\xi_i} = g_{\xi_i} \text{ or }h_{\xi_i} = \frac{\partial}{\partial \xi_i} g_{\xi_i},  
    \xi \in \mathbb R_+^m
    \right\} \subseteq \mathcal C^2(\mathbb S)
    \end{align}
    and define $G = G^{(1)} + G^{(2)}$ on $\mathcal D$ as in \eqref{eq:G1} and \eqref{eq:G2}. Then, $X$ solves the local $(\mathbb S_{\chi}, \delta_x, G, \mathcal D)$ martingale problem. Moreover, for $\xi, \eta, \tau \geq 0$, 
    \begin{enumerate}
        \item It holds
        $$ g_{ \xi} g_{ \eta} \in \mathcal D_0^2, \qquad S_t (g_{ \xi} g_{ \eta}) \in \mathcal D_t^2, \qquad G^{(1)}S_t(g_{ \xi} g_{ \eta})\in\mathcal D_t^4, \qquad G^{(2)}S_t(g_{ \xi} g_{ \eta})\in\mathcal D_t^4.$$
    \item There is $C>0$ (not depending on $\nu$) such that for all $t \in [0,\tau]$
    $$ |S_t (g_{ \xi} g_{ \eta})(x)| + |G^{(1)} S_t (g_{ \xi} g_{ \eta})(x)| + \nu |G^{(2)} S_t (g_{ \xi} g_{ \eta})(x)| \leq C \exp\Big( 4 \alpha \tau m_1(x)\Big). $$
    \item It holds
    \begin{align*}
    & f_{t, \xi, \eta} := \nu G^{(2)}S_t (g_{ \xi} g_{ \eta}) \in \mathcal D_t^4, \qquad S_s f_{t, \xi, \eta} \in \mathcal D_{s+t}^a, \\ & G^{(1)} S_s f_{t, \xi, \eta} \in \mathcal D_{s+t}^{a_1}, \qquad G^{(2)} S_s f_{t, \xi, \eta} \in \mathcal D_{s+t}^{a_2}   
    \end{align*}
    with $a, a_1, a_2$ from Corollary~\ref{cor2.4}.
    \item There is $C>0$ (not depending on $\nu$) such that for all $s, t \in [0,\tau]$ with $s+t \in [0, \tau]$
    $$ |S_s f_{t, \xi, \eta}(x)| + |G^{(1)} S_s f_{t, \xi, \eta}(x)| + \nu |G^{(2)} S_s f_{t, \xi, \eta}(x)| \!\leq\! C \exp\Big( \max(a, a_1, a_2) \alpha \tau m_1(x)\Big). $$
\end{enumerate}
\end{proposition}

\begin{proof}
Let $f := f_{\xi, {\zeta}}$ for some $\xi \in \mathbb R_+^m$ and ${\zeta \geq 0}$. For $\Phi \in \mathcal C^2(\mathbb R^{m} \times \mathbb R_{>0})$, given by $\Phi(x,y) := \frac{x_1\cdots x_m}{y^m}$, write $f = \Phi(h_{ \xi_1},...,h_{ \xi_m}; g_{\alpha {\zeta}})$ with $h_{\xi_i} = {g}_{\xi_i}$ or $h_{\xi_i} = \frac{\partial}{\partial {\xi_i}} {g}_{\xi_i}$. 

{
Let $\eta \in \{\xi_1,...,\xi_m, \alpha t\}$. Since $X$ solves the $(\mathbb S_\chi, \delta_x, G, \mathcal F)$ local martingale problem, has continuous paths, and $h_{\eta} \in \mathcal F$ is bounded, the process $(h_{\eta}(X_t))_{t\geq 0}$, is a continuous semi-martingale; see Proposition~\ref{martprob}.} Hence, by It\^o's Lemma, $(f(X(t)) - \int_0^t Gf(X(s)))_{t\geq 0}$ is a continuous local martingale. Since $f$ was fixed, we have shown that $X$ solves the $(\mathbb S_{\chi}, \delta_x, G, \mathcal D)$ local martingale problem. 

\noindent 1. Note that $g_0 = 1$, implying the first assertion. For the second, note that $S_t (g_{ \xi} g_{ \eta}) = (S_t g_{ \xi})(S_t g_{ \eta})$. The rest is a direct consequence of Corollary \ref{cor2.4}. 

\noindent 2. By 1., there is some $C$ (depending only on $\xi, \eta, \tau$) such that
$$ |S_t (g_{ \xi} g_{ \eta})(x)| + |G^{(1)} S_t (g_{ \xi} g_{ \eta})(x)| + \nu |G^{(2)} S_t (g_{ \xi} g_{ \eta})(x)| \leq \frac{C}{g_{\alpha t}(x)^4}.$$
Hence the result follows from \eqref{eq:jens}.

\noindent
3. As in 1., this is a consequence of Corollary~\ref{cor2.4}. 

\noindent
4. This follows along the same lines as 2.
\end{proof}

%\section{Proofs of Theorem~\ref{T1} and Corollary~\ref{cor}}

\section{Proofs of main results}

Recall that $G^{(1)}$ is the generator of the dynamical system (see \eqref{eq:G1}), and $S$ the corresponding semigroup, and $G^{(2)}$ is the resampling operator from \eqref{eq:G2}. Because of Proposition~\ref{P:locMartProb}, the weak solution $X$ of~\eqref{eq:ratchetSDE} solves the local martingale problem for the operator $G=G^{(1)} + G^{(2)}$ with $\mathcal D$ from \eqref{def:D}. 

\subsection{Proof of Theorem~\ref{T1}}
The proof proceeds in four steps; see also Figure~\ref{fig:proof-structure}. Throughout, we put $x:= {\rm Poi}(\theta)$ and fix $\xi, \eta \ge 0$ and $\tau > 0$.

\noindent\textbf{Step 1:} The process
$$
\left(
S_{\tau-t}(g_\xi g_\eta)(X(t))
-
\int_0^t
G^{(2)}S_{\tau-s}(g_\xi g_\eta)(X(s))\,ds
\right)_{0\leq t\leq\tau}
$$
is a martingale. In particular, since for our specific choice $x={\rm Poi}(\theta)$ we have 
\begin{equation}\label{fixedpoint}
S_\tau f (x) = f(x) \quad \mbox{ for all } f\in \mathcal D,
\end{equation}
this implies 
\begin{align}
  \label{eq:toshow2}
  \mathbb E_{x}\big((g_\xi g_\eta)(X(\tau))\big)
  &=
  (g_\xi g_\eta)(x)
  +
  \int_0^\tau
  \mathbb E_{x}\left(
  G^{(2)} S_{\tau-s}(g_\xi g_\eta)(X(s))
  \right)\,ds.
\end{align}

\emph{Proof:}
We apply Proposition~\ref{T:434appl} with $f=g_\xi g_\eta$. Therefore, we need to check the assumptions. Take $\mathcal D$ and $\mathcal D_t^m$ from  Proposition~\ref{P:locMartProb}, so we see that $X$ solves the local martingale problem for $\mathcal D$, and $f \in \mathcal D_0^2$, as well as $S_t f \in \mathcal D_t^2$, so (i) from Proposition~\ref{T:434appl} holds. Clearly, (ii) holds by  \eqref{eq:sol42}. For (iii), \eqref{eq:L1} is satisfied with $L(x) := C \exp\big(4\alpha \tau m_1(x)\big)$ with $C$ from Proposition~\ref{P:locMartProb}.2. Since $\nu \to \infty$ in Theorem~\ref{T1}, we have $2\tau < \nu$ wlog, so Proposition~\ref{P:bound} gives~\eqref{eq:L2}.

\noindent\textbf{Step 2:}  For $t \in [0,\tau]$ define
\begin{align}
    \label{eq:ftxieta}
    f_{t,\xi,\eta} := \nu \cdot G^{(2)}S_t (g_\xi g_\eta).
\end{align}
Then the process
$$ \Big(S_{t-s}f_{\tau-t,\xi,\eta}(X(s)) - \int_0^s G^{(2)} S_{t-r} f_{\tau-t,\xi,\eta}(X(r))dr\Big)_{0\leq s\leq t} $$
is a martingale. In particular, again because of~\eqref{fixedpoint}
\begin{align}\label{eq:G2G2}
    \mathbb E_x\left(G^{(2)}S_{\tau-t} (g_\xi g_\eta)(X(t))\right) \!-\! G^{(2)}S_{\tau-t} (g_\xi g_\eta)(x) \!=\! \frac 1 \nu \int_0^t \!\mathbb E_x\left(G^{(2)} S_{t-s} f_{\tau-t, \xi, \eta}(X(s))\right) ds. 
\end{align}
{\em Proof:} We apply the same procedure as in Step~1 with $g_\xi g_\eta \in \mathcal D_0^2$ replaced by $f_{\tau-t, \xi, \eta} \in \mathcal D_{\tau-t}^4$ (cf.\ Corollary~\ref{cor2.4}). For the applicability of Proposition~\ref{T:434appl}, we use Proposition~\ref{P:locMartProb}.3/4, leading to $L(x) := C \exp\big(\max(a, a_1, a_2) \alpha \tau m_1(x) \big)$ with $C, a, a_1, a_2$ from Proposition~\ref{P:locMartProb}.4 together with Proposition~\ref{P:bound}.

\noindent\textbf{Step 3:} 
\begin{align}
\label{eq:toshow4}
 \sup_{0 \leq t \leq\tau} \Big| \mathbb E_x\left(G^{(2)}S_{\tau-t} (g_\xi g_\eta)(X(t))\right) - G^{(2)}S_{\tau-t} (g_\xi g_\eta)(x) \Big| = \mathcal O \Big( \frac 1{\nu^2}\Big) \quad \mbox{ as } \nu \to \infty.
\end{align}

\noindent
{\em Proof:} We use \eqref{eq:G2G2} and Proposition~\ref{P:locMartProb}.4 in order find $C>0$ (depending on $\xi, \eta, \tau$, but not on $\nu$) such that the left hand side in \eqref{eq:toshow4} is bounded from above by  
\begin{align*}
    \frac 1 \nu \int_0^\tau \mathbb E_x \Big(\Big|G^{(2)} & S_{t-s} f_{\tau-t, \xi, \eta}(X(s)\Big|\Big)ds \Big) \\ & \leq \frac{C}{\nu^2} \mathbb E_x\Big(\sup_{0\leq t\leq \tau} \exp\big(\max(a, a_1, a_2) \alpha \tau m_1(X(t)) \big) \Big).
\end{align*}
Hence, the result follows from Proposition~\ref{P:bound}.

\noindent\textbf{Step 4:}  We now put Steps 1 and 3 together. For $x= {\rm Poi} (\theta)$ we have (cf.~\eqref{eq:hxiPoi} )
\begin{equation}\label{gxigeta}
    (g_\xi g_\eta)(x) = \exp(-\theta(2-e^{-\xi}-e^{-\eta})).
\end{equation}
Combined with ~\eqref{eq:toshow2} this implies the equality 
\begin{equation}\label{intermediate1}
\mathbb E_x((g_\xi g_\eta)(X(\tau))) =  \exp\Big( -\theta(2 - e^{-\xi} - e^{-\eta})\Big) + \int_0^\tau \mathbb E_x\left(G^{(2)} S_{\tau-t}(g_\xi g_\eta) (X(t))\right) dt
\end{equation}
From~\eqref{eq:toshow4} we obtain that the integral in~\eqref{intermediate1} equals
\begin{equation}\label{intermediate2}
\int_0^\tau G^{(2)}S_{\tau-t}(g_\xi g_\eta)(x)\, dt + \mathcal O\bigg(\frac 1{\nu^2}\bigg)
\end{equation}
For completing the proof of Theorem~\ref{T1} it remains to relate~\eqref{intermediate2} with the expression on the r.h.s. of the equality asserted in Theorem~\ref{T1}. This is achieved by multiplying~\eqref{eq:sol42Poi} by $g_{\xi}(x)g_{\eta}(x)$, expressing the resulting equality by $g_\xi$ and $g_\eta$ and again using~\eqref{gxigeta}. 

\subsection{Proof of Corollary~\ref{cor}}
Within the corollary, we will take derivatives of $g_\xi g_\eta$, and let $\eta\to\infty$ in places. We use 
\begin{align}\label{eq:mder}
    m_1(x) := \sum_{k=0}^\infty kx_k = - \frac{\partial g_\xi(x)}{\partial \xi}\Big|_{\xi = 0}.
\end{align}
and
\begin{align}\label{eq:mder2}
    x_k = \text{coefficient of } e^{- k \xi} \text{ in } g_\xi(x).
\end{align}
Moreover, 
\begin{align}
    m_2(x) := \sum_{k=0}^\infty k^2 x_k = \frac{\partial^2}{\partial \xi^2} g_\xi(x)\big|_{\xi = 0}, \quad \bar m_2(x) := \sum_{k=0}^\infty (k - m_1(x))^2 x_k = m_2(x) - m_1(x)^2 
\end{align}
and we define 
\begin{equation}
\label{eq:Gamma}
\begin{aligned}
    \Gamma(t, \xi, \eta) &:= 3 \exp\Big(\theta(1 - e^{-t})^2 \Big) - 2 \exp\Big(\theta(1 - e^{- (t + \xi)})(1 - e^{-t}) \Big) \\ & \qquad - 2 \exp\Big(\theta(1 - e^{-(t + \eta)})(1 - e^{-t}) \Big) + \exp\Big( \theta(1 - e^{-(t + \xi)})(1 - e^{-(t + \eta)})\Big),
\end{aligned}    
\end{equation}
which is the integrand in Theorem~\ref{T1}. 

\noindent\emph{Proof of \eqref{eq:cor1}:} Recall \eqref{eq:Ehxi}, and use \eqref{eq:mder} in order to obtain \begin{align*}
    \mathbb E\big(&m_1(X(\tau))\big) = -\mathbb E\left(\frac{\partial g_\xi(X(\tau))}{\partial \xi} \bigg|_{\xi = 0}\right)\\
&\overset{\triangle}= - \frac{\partial}{\partial \xi}  \mathbb E\left(g_\xi(X(\tau))\right)\bigg|_{\xi = 0}\\
    &= \theta - \frac 1{\nu \alpha} \int_0^{\tau \alpha} \frac{\partial}{\partial \xi} \exp\Big(\theta(1 - e^{-t})^2 \Big) - \exp\Big(\theta(1 - e^{-(t + \xi)})(1 - e^{-t}) \Big) \Big|_{\xi = 0} dt \\ & = \theta + \frac{1}{2\nu \alpha} \int_0^\tau \frac{\partial}{\partial s} \exp\Big( \theta(1 - e^{-\alpha s})^2 \Big) ds \\ & = \theta +  \frac{1}{2\nu \alpha} 
  \Big(\exp\big(\theta (1-e^{-\alpha \tau })^2\big) - 1\Big) + \mathcal O\Big(\frac{1}{\nu^2}\Big)
\end{align*}
We justify the interchange of derivative and expectation at $(\triangle)$. For $h>0$,
\begin{align*}
0 \leq -\frac{g_h(X_\tau)-g_0(X_\tau)}{h}
&=
\sum_{k=0}^\infty X_k(\tau)\frac{1-e^{-hk}}{h} \leq
\sum_{k=0}^\infty kX_k(\tau)
=
m_1(X_\tau).
\end{align*}

Since $\mathbb E(m_1(X_\tau))<\infty$, dominated convergence yields
$$
\lim_{h\downarrow0}
\mathbb E\left(
-\frac{g_h(X_\tau)-g_0(X_\tau)}{h}
\right)
=
\mathbb E\left(
-\left.\frac{\partial}{\partial\xi}g_\xi(X_\tau)\right|_{\xi=0}
\right).
$$

\noindent\emph{Proof of \eqref{eq:cor21}:}
Rearranging \eqref{eq:Ehxi} gives
\begin{align*}
    \mathbb E\big(g_\xi(X(\tau))\big) & = e^{-\theta(1 - e^{-\xi})} + \frac 1{\nu \alpha} \int_0^{\tau\alpha} \exp\Big(-\theta e^{-t}\Big) \Big(  \exp\Big(\theta(e^{-\xi} - e^{-t}(1 - e^{-t}))\Big)  \\ & \qquad \qquad \qquad \qquad \qquad - \exp\Big(\theta e^{-\xi}(1 - e^{-t} + e^{-2t})\Big)\Big)dt + \mathcal O\Big(\frac{1}{\nu^2}\Big).
\end{align*}
Viewing this as a power series in $e^{-\xi}$ and collecting terms of $e^{- k \xi}$ gives the first result. \eqref{eq:cor22} follows by setting $k=0$.

\noindent\emph{Proof of \eqref{eq:cor31}:}
The definition of $\Gamma$ from \eqref{eq:Gamma} implies
\begin{align*}
\Gamma(t, \xi, 0) & = \exp\Big(\theta(1 - e^{-t})^2 \Big) - \exp\Big(\theta(1 - e^{- (t + \xi)})(1 - e^{-t}) \Big).
\end{align*}
We compute
\begin{align*}
    \frac{\partial}{\partial \xi} \Gamma(t,\xi, 0)\Big|_{\xi=0} & = - \theta (1 - e^{-t}) e^{-t} \exp\Big(\theta(1 - e^{- t})^2 \Big) = - \frac 12 \frac{\partial}{\partial t} \exp\Big(\theta(1 - e^{-t})^2\Big), \\
    \frac{\partial^2}{\partial \xi^2} \Gamma(t,\xi, 0)\Big|_{\xi=0} & = - \theta (1 - e^{-t}) \frac{\partial}{\partial\xi} e^{-(t + \xi}) \exp\Big(\theta(1 - e^{- (t + \xi)})(1 - e^{-t}) \Big) \Big|_{\xi=0} \\ & = \theta(1 - e^{-t}) e^{-t} \Big(1 - \theta( 1 - e^{-t})e^{-t} \Big)\exp\Big(\theta(1 - e^{- t})^2 \Big), \\
    \frac{\partial^2}{\partial \xi \partial\eta} \Gamma(t,\xi,\eta)\Big|_{\xi=\eta=0} & = \theta e^{-t}\frac{\partial}{\partial \xi} (1 - e^{-(t + \xi)}) \exp\Big( \theta(1 - e^{-(t + \xi})(1 - e^{-t})\Big)\Big|_{\xi = 0} \\ & = 
    \big(\theta e^{-2t} + (\theta(1 - e^{-t})e^{-t})^2\big) \exp\Big(\theta(1 - e^{- t})^2 \Big).
\end{align*}
Moreover, 
\begin{align*}
    \frac{\partial^2}{\partial\xi^2} e^{-\theta(1- e^{-\xi})}\Big|_{\xi = 0} & = - \frac{\partial}{\partial\xi} \theta e^{-\xi} e^{-\theta(1- e^{-\xi})}\Big|_{\xi = 0} = \theta + \theta^2,\\
    \frac{\partial}{\partial\xi \partial\eta} e^{-\theta(2- e^{-\xi} - e^{-\eta})}\Big|_{\xi = \eta = 0} & = \theta^2. 
\end{align*}
Together, this gives
\begin{align*}
    \mathbb E&\big(\bar m_2(X(\tau))\big) \\ & = \Big(\frac{\partial^2}{\partial\xi^2} - \frac{\partial^2}{\partial\xi \partial \eta}\Big) \exp\Big( \!-\theta(2 - e^{-\xi} - e^{-\eta}\Big)\Big(1 + \frac{1}{\nu \alpha} \int_0^{\tau\alpha} \Gamma(t, \xi, \eta) dt \Big)\Big|_{\xi = \eta = 0} \!\! + \mathcal O\Big(\frac{1}{\nu^2}\Big) \\ & = \theta + \frac{1}{\nu \alpha} \int_0^{\tau\alpha} \big( \theta e^{-t}(1 - 2e^{-t})) - 2\theta^2(1 - e^{-t})e^{-t} \big)\exp\Big(\theta(1 - e^{-t})^2 \Big) dt  + \mathcal O\Big(\frac{1}{\nu^2}\Big) \\ & = \theta - \frac{1}{\nu \alpha} \int_0^{\tau \alpha} \theta \frac{\partial}{\partial t} e^{-t}(1 - e^{-t}) \exp\Big(\theta(1 - e^{-t})^2 \Big) dt + \mathcal O\Big(\frac{1}{\nu^2}\Big)\\ & = \theta - \frac{\theta}{\nu \alpha}e^{-\tau\alpha}(1 - e^{-\tau\alpha}) \exp\Big(\theta(1 - e^{-\tau\alpha})^2 \Big) + \mathcal O\Big(\frac{1}{\nu^2}\Big)
\end{align*}

\noindent\emph{Proof of \eqref{eq:cor32}:}
Note that
\begin{align*}
    \mathbb E\big(m_1(X(\tau))\big)^2 & = \theta^2 - \frac{2\theta}{\nu \alpha} \frac{\partial}{\partial\xi}\int_0^{\tau\alpha}  \Gamma(t,\xi,\eta) dt \Big|_{\xi = \eta = 0} + \mathcal O\Big(\frac{1}{\nu^2}\Big).
\end{align*}
So, we write
\begin{align*}
    \mathbb V(m_1(X(\tau))) & = \frac{\partial^2}{\partial\xi \partial \eta} \exp\Big( -\theta\Big(2 - e^{-\xi} - e^{-\eta}\Big)\Big)\Big(1 + \frac{1}{\nu \alpha} \int_0^{\tau\alpha} \Gamma(t, \xi, \eta) dt \Big)\Big|_{\xi = \eta = 0}  \\ & \qquad - \theta^2  + \frac{2\theta}{\nu \alpha} \frac{\partial}{\partial\xi} \int_0^{\tau\alpha} \Gamma(t,\xi,\eta)dt\Big|_{\xi=\eta = 0} +  \mathcal O\Big(\frac{1}{\nu^2}\Big)  \\ & = \frac{1}{\nu \alpha} \frac{\partial} {\partial\xi \partial \eta} \int_0^{\tau\alpha} \Gamma(t, \xi, \eta) dt\Big|_{\xi = \eta = 0} +  \mathcal O\Big(\frac{1}{\nu^2}\Big) 
\end{align*}
and the result follows. We can justify the change of limit and integral analogously to the proof of \eqref{eq:cor1}.

\noindent\emph{Proof of \eqref{eq:cor4}:}
We have 
\begin{align*}
&\operatorname{Cov}(m_1(X(\tau)), X_0(\tau)) = \lim_{\eta \to\infty} \mathbb E\big(g_\eta(X(\tau))\big) \frac{\partial}{\partial\xi} \mathbb E\big(g_\xi(X(\tau))\big)\Big|_{\xi=0}\\ & \qquad \qquad \qquad \qquad \qquad \qquad \qquad \qquad  - \lim_{\eta \to\infty} \frac{\partial}{\partial\xi} \mathbb E\big(g_\xi(X(\tau))g_\eta(X(\tau))\big)\Big|_{\xi=0} \\ & = \Big( e^{-\theta} + \frac{e^{-\theta}}{\nu \alpha} \int_0^{\tau\alpha} \Gamma(t, 0, \infty) dt \Big) \Big(- \theta + \frac{1}{\nu \alpha} \int_0^{\tau\alpha} \frac{\partial}{\partial\xi} \Gamma(t,\xi, 0) \Big|_{\xi = 0} dt \Big) \\ & \qquad \qquad \qquad \qquad - \frac{\partial}{\partial\xi} \Big( e^{-\theta(2 - e^{-\xi})}\Big(1  + \frac{1}{\nu \alpha} \int_0^{\tau\alpha} \Gamma(t, \xi, \infty) \Big|_{\xi = 0} dt \Big) \Big) + \mathcal O\Big(\frac{1}{\nu^2}\Big)
\\ & = -\theta e^{-\theta} + \frac{e^{-\theta}}{\nu \alpha} \int_0^{\tau\alpha} \Big(- \theta \Gamma(t,0,\infty) + \frac{\partial}{\partial\xi} \Gamma(t,\xi,0) \Big|_{\xi = 0}\Big) dt \\ & \qquad \qquad \qquad + \theta e^{-\theta} + \frac{e^{-\theta}}{\nu \alpha} \int_0^{\tau\alpha} \theta \Gamma(t,0,\infty) - \frac{\partial}{\partial\xi} \Gamma(t,\xi,\infty)\Big|_{\xi = 0} dt + \mathcal O\Big(\frac{1}{\nu^2}\Big)
 \\ & 
= \frac{e^{-\theta}}{\nu \alpha} \int_0^{\tau\alpha} \frac{\partial}{\partial\xi}\Big( \Gamma(t,\xi,0) - \Gamma(t,\xi,\infty)\Big) \Big|_{\xi = 0} dt + \mathcal O\Big(\frac{1}{\nu^2}\Big)
\\ & = \frac{e^{-\theta}}{\nu \alpha} \int_0^{\tau\alpha} 
\frac{\partial}{\partial\xi} \Big( \exp\Big( \theta(1 \! - e^{-(t + \xi)})(1 \!- e^{-t}) \Big) - \exp\Big( \theta(1 - e^{-(t + \xi)}) \Big) \Big)\Big|_{\xi = 0} dt + \mathcal O\Big(\frac{1}{\nu^2}\Big)
\\ & = \frac{e^{-\theta}}{\nu \alpha} \int_0^{\tau\alpha} 
\frac{\partial}{\partial t} \Big( \tfrac 12 \exp\Big( \theta(1 - e^{-t})^2 \Big) - \exp\Big( \theta(1 - e^{-t}) \Big) \Big)dt + \mathcal O\Big(\frac{1}{\nu^2}\Big)
\\ & = \frac{e^{-\theta}}{\nu \alpha} \Big( \tfrac 12 \exp\Big( \theta(1 - e^{-\tau\alpha})^2 \Big) - \exp\Big( \theta(1 - e^{-\tau\alpha})\Big) + \tfrac 12\Big) + \mathcal O\Big(\frac{1}{\nu^2}\Big)
\end{align*}

\subsection{Calculations for Remark~\ref{rem:sanity}}
\label{ss:sanity}
For \eqref{eq:sanity}, taking the time derivative in \eqref{eq:cor1}  gives
$$
\frac{d}{d\tau}\mathbb E\big(m_1(X(\tau))\big) = \frac1\nu \theta e^{-\alpha\tau} \left(1-e^{-\alpha\tau}\right) \exp\left(\theta\left(1-e^{-\alpha\tau}\right)^2\right) + \mathcal O\left(\frac{1}{\nu^2}\right).
$$
where we use that the order of the error term is stable under differentiation. Using $\lambda=\alpha\theta$ and \eqref{eq:cor31}, we obtain
\begin{align*}
\lambda-\alpha\mathbb E\big(\bar m_2(X(\tau))\big) &= \alpha\theta - \alpha \Big( \theta - \frac{\theta}{\nu \alpha} e^{-\alpha\tau} \Big(1-e^{-\alpha\tau}\Big) \exp\Big(\theta\Big(1-e^{ - \alpha\tau}\Big)^2\Big) \Big) + \mathcal O\Big(\frac{1}{\nu^2}\Big) \\ &= \frac 1\nu \theta e^{- \alpha\tau} \Big(1-e^{- \alpha\tau}\Big) \exp\Big(\theta\Big(1-e^{-\alpha\tau}\Big)^2\Big) + \mathcal O\Big(\frac{1}{\nu^2}\Big).
\end{align*}
Thus both sides of \eqref{eq:sanity} agree up to a summand $\mathcal O(\nu^{-2})$.

For \eqref{eq:sanity2}, taking the time derivative in \eqref{eq:cor22}  gives
\begin{align*}
\frac{d}{d\tau}\mathbb E\big(X_0(\tau)\big) & = - \frac1{\nu} \exp\Big(-\theta e^{-\alpha\tau}\Big) \Big( 1 - \exp\Big(- \theta e^{-\alpha\tau}(1 - e^{-\tau\alpha})\Big)  \Big) + \mathcal O\left(\frac{1}{\nu^2}\right)    
\end{align*}
where we use that the order of   the error term is stable under differentiation. Then, we compute
\begin{align*}
    \mathbb E\big(m_1(X(\tau)) X_0(\tau)\big) & = - \lim_{\eta \to\infty} \frac{\partial}{\partial\xi} \mathbb E\big(g_\xi(X(\tau))g_\eta(X(\tau))\big) \Big|_{\xi = 0} \\ & = - \frac{\partial}{\partial\xi} \Big( e^{-\theta(2 - e^{-\xi})}\Big(1  + \frac{1}{\nu \alpha} \int_0^{\tau\alpha} \Gamma(t, \xi, \infty) dt \Big) \Big)\Big|_{\xi = 0} + \mathcal O\Big(\frac{1}{\nu^2}\Big) \\ & = \theta e^{-\theta} + \frac{e^{-\theta}}{\nu \alpha} \int_0^{\tau\alpha} \theta \Gamma(t,0,\infty) - \frac{\partial}{\partial\xi} \Gamma(t,\xi,\infty)\Big|_{\xi = 0} dt + \mathcal O\Big(\frac{1}{\nu^2}\Big)
    \\ & = \theta e^{-\theta} + \frac{\theta e^{-\theta}}{\nu \alpha} \Big( \int_0^{\tau\alpha} \Big( \exp\Big(\theta(1 - e^{-t})^2\Big) - \exp\Big(\theta(1 - e^{-t}) \Big) \Big) dt \\ & \quad - \! \int_0^{\tau\alpha} \frac{\partial}{\partial t} \Big( \exp\Big(\theta(1 - e^{-t})\Big) - \exp\Big(\theta(1 - e^{-t})^2 \Big)\Big) dt \Big) + \mathcal O\Big(\frac{1}{\nu^2}\Big)
\end{align*}
Now we use \eqref{eq:cor4} to obtain
\begin{align*}
\mathbb E\big( & (\alpha m_1(X(\tau)) - \lambda)X_0(\tau)\big) = \alpha \mathbb E\big((m_1(X(\tau)) - \theta)X_0(\tau)\big) \\ 
& 
= \frac{e^{-\theta}}{\nu} \Big( - \Big( \exp\Big(\theta(1 - e^{-\tau\alpha})\Big) - \exp\Big(\theta(1 - e^{-\tau\alpha})^2 \Big)\Big) \\ & \qquad \qquad + \int_0^{\tau\alpha} \theta \Big( \exp\Big(\theta(1 - e^{-t})^2\Big) - \exp\Big(\theta(1 - e^{-t}) \Big) \Big) dt
\Big) \\ &  \qquad \qquad \qquad \qquad - \frac{\theta}{\nu} \int_0^{\tau \alpha} \exp\Big(-\theta e^{-t}\Big)\Big(1 - \exp\Big(-\theta e^{-t}(1 - e^{-t})\Big)\Big) dt
\\ &
= - \frac{1}{\nu} \Big( \Big( \exp\Big(- \theta e^{-\tau\alpha}\Big) - \exp\Big(-\theta e^{-\tau\alpha}(2  - e^{-\tau\alpha}) \Big)\Big) \\ & \qquad \qquad + \theta \int_0^{\tau\alpha} \Big( \exp\Big(-\theta e^{-t}(2 - e^{-t})\Big) - \exp\Big(-\theta e^{-t} \Big) \Big) dt
\Big) \\ &  \qquad \qquad \qquad \qquad - \frac{\theta}{\nu} \int_0^{\tau \alpha} \exp\Big(-\theta e^{-t}\Big)\Big(1 - \exp\Big(-\theta e^{-t}(1 - e^{-t})\Big) \Big)dt
\\ & 
= - \frac1{\nu} \exp\Big(-\theta e^{-\alpha\tau}\Big) \Big( 1 - \exp\Big(- \theta e^{-\alpha\tau}(1 - e^{-\tau\alpha})\Big)  \Big) + \mathcal O\left(\frac{1}{\nu^2}\right).
\end{align*}
Thus both sides of \eqref{eq:sanity2} agree up to a summand $\mathcal O(\nu^{-2})$.

\subsection{Calculations for Section~\ref{ss:interpretation}}
\label{rem:interpretation}
Note that \eqref{eq:cor1},  \eqref{eq:cor31} and \eqref{eq:cor4} come without any integral in the $\mathcal O()$-term, so the calculation is more direct. Moreover, we have by our choice of $\sigma, \tau$ (note $\alpha \sigma = c \tfrac{\log\theta}{\theta} \downarrow 0$)
\begin{align}
    \label{eq:sigmatauapp}
    & e^{-\alpha \sigma} \to 1, \quad \theta(1 - e^{-\alpha\sigma}) = c \log\theta \to \infty, \quad \theta(1 - e^{-\alpha\sigma})^2 = \frac{c^2 \log^2\theta}{\theta} \downarrow 0, \quad \theta e^{-\alpha \tau} \downarrow 0.
\end{align}
For \eqref{eq:cor1app}/\eqref{eq:cor14aapp}/\eqref{eq:cor4app}, the asymptotics arise directly by \eqref{eq:cor1}/\eqref{eq:cor31}/\eqref{eq:cor4} together with \eqref{eq:sigmatauapp}. 

For \eqref{eq:cor2app}, we write for  $\sigma := \tfrac c\lambda \log\theta$, to leading order,
\begin{align*}
    \int_0^{\sigma \alpha} &  \exp\Big(-\theta e^{-t}\Big)\Big(1 - \exp\Big(-\theta e^{-t}(1 - e^{-t})\Big) dt \stackrel{s:=e^{-t}} =
    e^{-\theta} \int_{e^{-\sigma \alpha}}^1 \frac{e^{\theta(1 - s)} - e^{\theta(1 - s)^2}}{s}ds \\ & 
    = e^{-\theta} \int_0^{1 - e^{-\sigma\alpha}} \frac{e^{\theta s} - e^{\theta s^2}}{1 - s} ds
    \approx \int_{e^{-\sigma\alpha}}^1 e^{-\theta s} ds = \frac{e^{-\theta}}{\theta} \Big(e^{\theta (1 - e^{-\sigma \alpha})} - 1\Big) 
    \approx e^{-\theta} \theta^{c-1},
\end{align*}
which shows the first assertion. For $\tau \gg \tfrac 1 \alpha \log\theta$, the second assertion follows with
\begin{align*}
    \int_0^{\tau \alpha} &  \exp\Big(-\theta e^{-t}\Big)\Big(1 - \exp\Big(-\theta e^{-t}(1 - e^{-t})\Big) dt \stackrel{s:=\theta e^{-t}} =
    \int_{\theta e^{-\tau\alpha}}^\theta \frac{e^{-s} - e^{-2s + s^2/\theta}}{s} ds  \\ & \approx \int_0^\theta \frac{e^{-s} - e^{-2s}}{s} ds \approx \int_0^\infty \int_1^2 e^{-as} da ds = \int_1^2 \int_0^\infty  e^{-as} ds da = \int_1^2 \frac 1a da = \log 2.
\end{align*}
Let us turn to \eqref{eq:cor3bapp}. Here, the terms $\theta(1 - e^{-t})^2 \downarrow 0$ by \eqref{eq:sigmatauapp}, hence
\begin{align*}
    \int_0^{\sigma \alpha} & \theta e^{-2t}\big(1 + \theta(1 - e^{-t})^2 \big) \exp\Big(\theta(1 - e^{- t})^2 \Big) dt \approx \int_0^{\sigma\alpha} \theta e^{-2t} dt \approx c \log\theta,
\end{align*}
showing the first assertion. For the second, 
\begin{align*}
    \int_0^{\tau \alpha} & \theta e^{-2t}\big(1 + \theta(1 - e^{-t})^2 \big) \exp\Big(\theta(1 - e^{- t})^2 \Big) dt \\ & \stackrel{z = \theta e^{-t}} \approx e^\theta \int_0^\infty \tfrac{z^2}{\theta} (1 + \theta) \exp\big(-2 z  \big) dz \approx e^\theta \int_0^\infty z e^{-2z} dz = \tfrac 14 e^\theta. 
\end{align*}

\begin{appendix}

\section{Martingales along the Duhamel formula}
\label{app:duhamel}
The main result of this section, Proposition~\ref{T:434appl}, provides the martingales that are used for the proof of Theorem~\ref{T1}, see also the explanation in Remark~\ref{rem:overview}. Because of its independent interest we will phrase Proposition~\ref{T:434appl} in a framework that is more general than that of paper's main part. In Remark~\ref{rem:duhamel} we will point to the connection between the martingales appearing in~\eqref{eq:duhamelmart} and the Duhamel formula for semigroups.

\begin{proposition}\label{T:434appl}
Let $(E,r)$ be a complete and separable metric space, $\mathcal D \subseteq \mathcal M(E)$ linear, and $G^{(1)}, G^{(2)} : \mathcal D \to \mathcal M(E)$. Assume $G^{(1)}$ is the generator in the sense of bounded pointwise limits of the semigroup $S$ (i.e.\ $(S_{h}f(x) - f(x))/h \xrightarrow{h\downarrow 0} G^{(1)}  f(x)$ boundedly for all $f \in \mathcal D$, $t\geq 0$, and $x\in E$). 

Let $X$ be a solution of the $(G^{(1)} + G^{(2)}, \mathcal D)$ local martingale problem with càdlàg paths, and $\tau \geq 0$. Fix $f \in \mathcal D$ and assume
\begin{itemize}
    \item[(i)] $S_{t} f \in \mathcal D$ and continuous, for all $t \in [0,\tau]$,
    \item[(ii)] $t \mapsto G^{(2)} S_{t}f(x)$ is right-continuous for all $x$ on $[0,\tau]$,
    \item[(iii)] there exist $L \in \mathcal M(E)$, such that for all $t \in [0, \tau]$ and $x\in E$ 
\begin{align}
\label{eq:L1}
    &|S_{t} f(x) | + |G^{(1)} S_{t}f(x)| + |G^{(2)} S_{t}f(x)| \leq L(x),
\end{align}
and
\begin{align}
    \label{eq:L2}
    & \mathbb E\big(\sup_{0\leq t \leq \tau} L(X(t))\big) < \infty.
\end{align}
\end{itemize}
Then
\begin{equation} \label{eq:duhamelmart}
  M_f(t) := S_{\tau-t} f(X(t)) - \int_0^t G^{(2)} S_{\tau-s} f(X(s)) \, ds,
  \qquad 0 \le t \le \tau
\end{equation}
is a martingale.
%\begin{align}
%\label{eq:mart}
%\Big(S_{\tau - t} f(X(t)) - \int_0^t G^{(2)} S_{\tau -s} f(X(s))ds\Big)_{0 \leq t \leq \tau}     
%\end{align}
%is a martingale.   
\end{proposition}

\begin{remark}\label{rem:duhamel} 
For the martingale $M_f$ as in \eqref{eq:duhamelmart}, with $X(0) =x$, the equality
  $\mathbb E(M_f(\tau)) = \mathbb E(M_f(0))$ translates into
  \begin{equation}\label{duha}
    \mathbb E_x(f(X(\tau))) - \int_0^\tau \mathbb E_x(G^{(2)} S_{\tau-s} f(X(s))) \, ds = S_\tau f(x).
 \end{equation}
This relates the semigroup of $X$ (associated with the operator $G^{(1)} + G^{(2)}$) with the semigroup $S$  (generated by $G^{(1)}$), and thus~\eqref{duha} can be seen as an instance of \emph{Duhamel's formula}. See e.g.\ Section~3.1, eq.~(1.2) in \cite{pazy}. 

Here is (the core of) an argument for the martingale property of $M_f$.
 Since we know that $X$ satisfies a martingale problem
involving the operator $G^{(1)} + G^{(2)}$, it is fair to hope that
$$\int_0^t \left( \frac{\partial}{\partial s} + G^{(1)} + G^{(2)} \right)
S_{\tau-s} f(X(s)) \, ds = \int_0^t G^{(2)} S_{\tau-s} f(X(s)) \, ds$$ 
is the compensator whose subtraction from $S_{\tau-t} f(X(t))$ yields a martingale, namely $M_f$. 
Proposition~\ref{T:434appl} puts this reasoning on firm grounds. 
  \end{remark}

\medskip

For the proof of Proposition~\ref{T:434appl}, we give a localized form of Lemma 4.3.4 (a) from \cite{EthierKurtz1986}, here with a numbering of equations chosen parallel to the numbering of their analoga in~\mbox{\cite[p. 176-177]{EthierKurtz1986}}.   

\begin{lemma}[Stopping-time variant of \cite{EthierKurtz1986}, Lemma 4.3.4]
\label{l:434}
Let $X$ be a measurable stochastic process on a probability space $(\Omega, \mathcal F, \mathbb P)$ with càdlàg paths in the complete and separable metric space $(E,r)$, adapted to the filtration $\{\mathcal G_t\}$. Let $\sigma$ be a $\{\mathcal G_t\}$-stopping time, and
$u, v : [0, \infty) \times E \times \Omega \to \mathbb R$ be $\mathcal B[0,\infty) \times \mathcal B(E) \times \mathcal F$-measurable, and let $w : [0, \infty) \times E \times \Omega \to \mathbb R$ be $\mathcal B[0,\infty) \times \mathcal B(E) \times \mathcal F$-measurable. Assume that $x\mapsto u(t, x, \omega)$ is continuous for fixed $t$ and $\omega$.

Suppose moreover there exists an integrable random variable $Y$, such that, almost surely,
\begin{equation}
\label{eq:dom}
  |u(s, X_{s \wedge \sigma})| + 
  |v(s, X_{s \wedge \sigma})| +
  |w(s, X_{s \wedge \sigma})| \leq Y
  \tag{D}
\end{equation}
for all $s \in [0, \infty)$. Suppose further that for every $t_2 > t_1 \geq 0$,
\begin{align}
  \mathbb E\big(u(t_2 \wedge \sigma, X_{t_2 \wedge \sigma}) - u(t_1 \wedge \sigma, X_{t_2 \wedge \sigma}) \mid \mathcal G_{t_1}\big)
  &= \mathbb E\Big(\int_{t_1 \wedge \sigma}^{t_2 \wedge \sigma}
      v(s, X_{t_2 \wedge\sigma})\, ds \,\Big|\, \mathcal G_{t_1} \Big).
  \label{eq:310prime}
  \tag{3.10$'$}\\[4pt]  
  \mathbb E\big(u(t_1\wedge \sigma, X_{t_2 \wedge \sigma}) - u(t_1\wedge \sigma, X_{t_1 \wedge\sigma}) \mid \mathcal G_{t_1}\big)
  & = \mathbb E\Big(\int_{t_1 \wedge \sigma}^{t_2 \wedge \sigma}
      w(t_1, X(s))\, ds \,\Big|\, \mathcal G_{t_1} \Big),
  \tag{3.11$'$}
\label{eq:311prime}
\end{align}
Moreover, 
\begin{equation}
  \lim_{\delta \to 0+}
    \mathbb E\Big(\, |w(t-\delta,  X_{t\wedge\sigma}) - w(t, X_{t\wedge\sigma})| \,\Big)
  = 0, \qquad t > 0.
  \tag{3.12$'$}
\label{eq:312prime}
\end{equation}
Then
\begin{equation}
\label{eq:316prime}
  M_t \;:=\; u(t \wedge \sigma, X_{t \wedge\sigma})
  - \int_0^{t \wedge \sigma}
    ( v(s, X(s)) + w(s, X(s)) )\, ds,
  \qquad t \geq 0,
  \tag{3.16$'$}
\end{equation}
is a $\{\mathcal G_t\}$-martingale.
\end{lemma}
%\noindent
%{\bf Step 4:} 
%\begin{align*}
%N|\mathbb E_x[f(X(t))] - S_tf(x)| & = N\int_0^t \mathbb E[G^{(2)}S_{t-s}f(X(s))]ds \\ & \leq \mathbb E\Big[ \Big] ds
%\end{align*}

\begin{proof}
The argument follows as in the proof of Lemma~4.3.4 in Ethier and
Kurtz (1986), with deterministic times replaced by their minimum with
$\sigma$ and global boundedness replaced by the trajectory domination~\eqref{eq:dom}.
\end{proof}

\begin{proof}[Proof of Proposition \ref{T:434appl}]
    We use Lemma \ref{l:434} above and set 
    $$ u(s,x) := S_{\tau-s}f(x), \qquad v(s,x) := -G^{(1)} S_{\tau-s}f(x), \quad w(s,x) := (G^{(1)}+G^{(2)}) S_{\tau-s}f(x).$$
    Note that \eqref{eq:dom} holds with $Y := \sup_{0\leq t\leq \tau} L(X(t))$. By assumption, $S_{\tau-s}f \in \mathcal D$ for all $s \in [0,\tau]$, $s \mapsto w(s,x)$ is left-continuous on $[0,\tau]$ (i.e.\ \eqref{eq:312prime} holds), and
    \begin{equation*}\label{vw} v(s,x) + w(s,x) = G^{(2)} S_{\tau-s}f(x).
    \end{equation*}
    Introduce the stopping times $$\sigma_n := \tau \wedge \inf\{t \geq 0: L(X(t)) > n\} \uparrow \tau$$
    (since $\sup_{s\leq\tau} L(X(s)) < \infty$ almost surely by \eqref{eq:L2}). Consequently, by \eqref{eq:L1},
    $$ |u(s,X_{s\wedge \sigma_n})|,\ |v(s,X_{s\wedge \sigma_n})|,\ |w(s,X_{s\wedge \sigma_n})| \leq C n \quad \text{for all } s\in [0,\tau].$$
    For any $0 \le t_1 \le t_2 \le \tau$ and all realisations of $\sigma_n$ we find, using the assumption that $G^{(1)}$ is the generator of the semigroup $S$, that
    \begin{align*}
        u(t_2 \wedge \sigma_n, X_{t_2}) - u(t_1 \wedge \sigma_n, X_{t_2}) = \int_{t_1 \wedge \sigma_n}^{t_2 \wedge \sigma_n} v(s,X_{t_2}) ds,
    \end{align*}
    i.e.\ \eqref{eq:310prime} holds. 
    Moreover, since $X$ solves the local $(G^{(1)} + G^{(2)}, \mathcal D)$ martingale problem, for any $0\leq t_1 \leq t_2 \leq \tau$,
    \begin{align*}
        \mathbb E\big(u(t_1, X_{t_2\wedge \sigma_n}) - u(t_1, X_{t_1\wedge \sigma_n}) | \mathcal G_{t_1}\big) & = \mathbb E \Big( \int_{t_1 \wedge \sigma_n}^{t_2 \wedge \sigma_n} (G^{(1)} + G^{(2)})u(t_1, X(s)) ds \Big| \mathcal G_{t_1}\Big),
    \end{align*}
    i.e.\ \eqref{eq:311prime} holds.  Therefore, we conclude from Lemma~\ref{l:434} that
    $$ M_{t\wedge \sigma_n} := S_{\tau - (t\wedge \sigma_n)} f(X_{t\wedge\sigma_n}) - \int_0^{t\wedge \sigma_n} G^{(2)} S_{\tau-s} f(X(s)) ds $$
    is a martingale. From \eqref{eq:L1} and \eqref{eq:L2}, we know that $|M_{t\wedge\sigma_n}| \leq (1+\tau)\cdot Y$. Therefore, by dominated convergence for conditional expectations, 
    $$ \mathbb E(M_t \mid \mathcal G_s) = \lim_n \mathbb E(M_{t\wedge\sigma_n} \mid \mathcal G_s) = \lim_n M_{s\wedge\sigma_n} = M_s$$
    almost surely, which proves the claim. 
\end{proof}

\section{Table of simulation results} \label{simtable}

\noindent
Table~\ref{tab1} shows some more details on the simulation results, which are the basis for Figures~\ref{fig1} and~\ref{fig2}.

\begin{table}[H]
\centering
\caption{\label{tab1} Simulated values of first clicks for population size $N=10^4$ in $10^3$ simulations per parameter combination, when starting the system in Poi$(\theta)$. ``--'' marks parameter sets where either not all paths showed a click in $10^6$ generations. This table is also available as an \href{https://pfaffelh.github.io/ratchet_figures/}{interactive version}.}
\setlength{\tabcolsep}{4pt}
\begin{adjustbox}{max width=\textwidth, max totalheight=0.9\textheight, center}
\scriptsize
\begin{tabular}{c lccccccccc}
\toprule
$\psi \backslash \delta$ &  & $0.1$ & $0.2$ & $0.3$ & $0.4$ & $0.5$ & $0.6$ & $0.7$ & $0.8$ & $0.9$ \\
\midrule
 & $Ns$ & 1.758 & 4.417 & 11.09 & 27.87 & 70 & 175.8 & 441.7 & 1109 & 2787 \\
 & $Nu$ & 1.619 & 8.136 & 30.65 & 102.7 & 322.4 & 971.7 & 2848 & 8175 & 2.31e+04 \\
$0.7$ & $N\,e^{-(u/s)}$ & 3981 & 1585 & 631 & 251.2 & 100 & 39.81 & 15.85 & 6.31 & 2.512 \\
 & $t_0$ & 2.585e+04 & 7602 & 2922 & 1079 & 426.8 & 166.4 & 64.38 & 24.66 & 9.278 \\
 & $t_0\,s/\log(u/s)$ & $<0$ & 5.496 & 3.189 & 2.305 & 1.956 & 1.712 & 1.526 & 1.370 & 1.222 \\
\midrule
 & $Ns$ & 2.512 & 6.31 & 15.85 & 39.81 & 100 & 251.2 & 631 & 1585 & 3981 \\
 & $Nu$ & 2.314 & 11.62 & 43.79 & 146.7 & 460.5 & 1388 & 4068 & 1.168e+04 & 3.3e+04 \\
$1$ & $N\,e^{-(u/s)}$ & 3981 & 1585 & 631 & 251.2 & 100 & 39.81 & 15.85 & 6.31 & 2.512 \\
 & $t_0$ & 2.537e+04 & 8176 & 2805 & 1088 & 412.3 & 155 & 58.76 & 23 & 8.59 \\
 & $t_0\,s/\log(u/s)$ & $<0$ & 8.444 & 4.375 & 3.320 & 2.700 & 2.278 & 1.989 & 1.825 & 1.617 \\
\midrule
 & $Ns$ & 3.265 & 8.202 & 20.6 & 51.75 & 130 & 326.5 & 820.2 & 2060 & 5175 \\
 & $Nu$ & 3.008 & 15.11 & 56.93 & 190.7 & 598.7 & 1805 & 5288 & 1.518e+04 & 4.29e+04 \\
$1.3$ & $N\,e^{-(u/s)}$ & 3981 & 1585 & 631 & 251.2 & 100 & 39.81 & 15.85 & 6.31 & 2.512 \\
 & $t_0$ & 2.822e+04 & 8215 & 2768 & 1037 & 428 & 150.8 & 57.72 & 23.07 & 8.302 \\
 & $t_0\,s/\log(u/s)$ & $<0$ & 11.031 & 5.610 & 4.116 & 3.644 & 2.881 & 2.540 & 2.380 & 2.031 \\
\midrule
 & $Ns$ & 5.024 & 12.62 & 31.7 & 79.62 & 200 & 502.4 & 1262 & 3170 & 7962 \\
 & $Nu$ & 4.627 & 23.25 & 87.58 & 293.3 & 921 & 2776 & 8136 & 2.336e+04 & 6.6e+04 \\
$2$ & $N\,e^{-(u/s)}$ & 3981 & 1585 & 631 & 251.2 & 100 & 39.81 & 15.85 & 6.31 & 2.512 \\
 & $t_0$ & 3.624e+04 & 9328 & 2914 & 1093 & 416.1 & 164.5 & 65.45 & 21.64 & 8.497 \\
 & $t_0\,s/\log(u/s)$ & $<0$ & 19.269 & 9.088 & 6.672 & 5.449 & 4.834 & 4.432 & 3.434 & 3.199 \\
\midrule
 & $Ns$ & 12.56 & 31.55 & 79.24 & 199.1 & 500 & 1256 & 3155 & 7924 & -- \\
 & $Nu$ & 11.57 & 58.11 & 219 & 733.3 & 2303 & 6941 & 2.034e+04 & 5.839e+04 & -- \\
$5$ & $N\,e^{-(u/s)}$ & 3981 & 1585 & 631 & 251.2 & 100 & 39.81 & 15.85 & 6.31 & -- \\
 & $t_0$ & 1.536e+05 & 2.598e+04 & 7041 & 2266 & 796 & 283.7 & 109.4 & 57.24 & -- \\
 & $t_0\,s/\log(u/s)$ & $<0$ & 134.148 & 54.900 & 34.595 & 26.063 & 20.844 & 18.511 & 22.713 & -- \\
\midrule
 & $Ns$ & -- & -- & 158.5 & 398.1 & 1000 & 2512 & 6310 & -- & -- \\
 & $Nu$ & -- & -- & 437.9 & 1467 & 4605 & 1.388e+04 & 4.068e+04 & -- & -- \\
$10$ & $N\,e^{-(u/s)}$ & -- & -- & 631 & 251.2 & 100 & 39.81 & 15.85 & -- & -- \\
 & $t_0$ & -- & -- & 5.639e+04 & 1.38e+04 & 4142 & 1453 & 991.8 & -- & -- \\
 & $t_0\,s/\log(u/s)$ & -- & -- & 879.267 & 421.433 & 271.205 & 213.564 & 335.783 & -- & -- \\
\bottomrule
\end{tabular}
\end{adjustbox}
\end{table}
\end{appendix}

\bigskip

\noindent
\textbf{Funding} \\
\noindent
CSH is funded by the Deutsche Forschungsgemeinschaft (DFG, German Research Foundation) – Project-ID 499552394 – SFB Small Data. \\ 
\noindent
\textbf{Data availability} \\
\noindent
The python scripts to perform the simulation studies are available on \href{https://doi.org/10.5281/zenodo.20574536}{Zenodo} (\url{https://doi.org/10.5281/zenodo.20574536}).\\
\noindent
\textbf{Declaration of Conflicts of Interest} \\
The authors declare that there are no conflicts of interest.
\\
\textbf{Use of Generative-AI Tools Declaration} \\
The authors declare they have used ChatGPT as well as Claude for coding the simulations,  figures and tables in Section~\ref{ss:interpretation}.

\bibliographystyle{imsart-number}
\bibliography{bibliography}

%%  MR numbers will be added by VTeX.                      %%
%%                                                         %%
%%  Use \cite{...} to cite references in text.             %%
%%                                                         %%
%%%%%%%%%%%%%%%%%%%%%%%%%%%%%%%%%%%%%%%%%%%%%%%%%%%%%%%%%%%%%

%% if your bibliography is in bibtex format, uncomment commands:
%\bibliographystyle{imsart-number} % Style BST file (imsart-number.bst or imsart-nameyear.bst)
%\bibliographystyle{chicago}

%\bibliography{bibliography}       % Bibliography file (usually '*.bib')

\end{document}